\documentclass{birkjour}

\usepackage{amssymb,amsmath,latexsym,amsthm,amsfonts}
\usepackage{blindtext}
\usepackage{todonotes}
\usepackage{esint,dsfont}

\usepackage{hyperref}

\def\R{{\mathbb R}}
\def\N{\mathbb{N}}
\def\C{\mathbb{C}}

\newtheorem{prop}{\bf Proposition}[section]
\newtheorem{thm}[prop]{\bf Theorem}
\newtheorem{cor}[prop]{\bf Corollary}
\newtheorem{lem}[prop]{\bf Lemma}
\newtheorem{rmk}[prop]{\it Remark}

\begin{document}

\title[Positive definite radial kernels]{Positive definite radial kernels on\\ homogeneous trees and products}

%%    Information for first author
%\author{ }
\author[I. Vergara]{Ignacio Vergara}
\address{Institute of Mathematics of the Polish Academy of Sciences,\\ ul. \'Sniadeckich 8, 00-656 Warszawa, Poland.}
%\curraddr{Department of Mathematics and Statistics,
%Case Western Reserve University, Cleveland, Ohio 43403}
\email{ivergara@impan.pl}
%%    \thanks will become a 1st page footnote.
%%    Information for second author
%\author{Author Two}
%\address{Mathematical Research Section, School of Mathematical Sciences,
%Australian National University, Canberra ACT 2601, Australia}

%    General info
%
%\date{\today}

\begin{abstract}
We give a new proof of a classical result which provides a one-to-one correspondence between positive definite radial kernels on a homogeneous tree and finite Borel measures on the interval $[-1,1]$. Our methods allow us to find a new characterisation in terms of positive trace-class operators on $\ell_2$. Furthermore, we extend both characterisations to finite products of homogeneous trees. The proof relies on a formula for the norm of radial Schur multipliers, in the spirit of Haagerup--Steenstrup--Szwarc, and a variation of the Hamburger moment problem.
\end{abstract}

\maketitle

\section{Introduction}

Let $X$ be a nonempty set. We say that a function $\psi:X\times X\to\C$ is positive definite (or that it is a positive definite kernel on $X$) if, for every finite family $x_1,...,x_n\in X$ and every $z_1,...,z_n\in\C$, we have
\begin{align}\label{ineq_pd}
\sum_{i,j=1}^n\psi(x_i,x_j)\overline{z_i}z_j \geq 0.
\end{align}

Positive definite kernels are particular cases of Schur multipliers. Let $\ell_2(X)$ be the Hilbert space of square-summable, complex-valued functions on $X$. A function $\psi:X\times X\to\C$ is said to be a Schur multiplier on $X$ if, for every bounded operator $T\in\mathcal{B}(\ell_2(X))$, there exists another bounded operator $M_\psi(T)\in\mathcal{B}(\ell_2(X))$ whose matrix coefficients are given by
\begin{align}\label{def_Mpsi}
\langle M_\psi(T)\delta_y,\delta_x\rangle = \psi(x,y) \langle T\delta_y,\delta_x\rangle,\quad  \forall x,y\in X.
\end{align}
In this case, $M_\psi$ defines a completely bounded map on $\mathcal{B}(\ell_2(X))$, and we denote by $\|\psi\|_{cb}$ its completely bounded norm, which coincides with its norm as a bounded operator on $\mathcal{B}(\ell_2(X))$. See \cite[Chapter 5]{Pis2} for more details.

Positive definite kernels can be characterised by the complete positivity of the associated map \eqref{def_Mpsi}.

\begin{prop}\label{Prop_cp_pd}
Let $X$ be a non-empty set and let $\psi:X\times X\to\C$ be a function which is constant on the diagonal. Then $\psi$ is positive definite if and only if the map $M_\psi$ given by \eqref{def_Mpsi} is completely positive.
\end{prop}

For a proof of this fact, see e.g. \cite[Theorem D.3]{BroOza}. For more on completely positive maps, see \cite[\S 1.5]{BroOza}. 

If now $X$ is (the set of vertices of) a connected graph, we say that $\psi:X\times X\to\C$ is radial if it depends only on the distance between each pair of vertices. In other words, if there exists $\varphi:\N\to\C$ such that
\begin{align*}
\psi(x,y)=\varphi(d(x,y)),\quad\forall x,y\in X.
\end{align*}
Here the distance $d(x,y)$ is given by the length of the shortest path between $x$ and $y$. We can thus define the following subsets of $\ell_{\infty}(\N)$:
\begin{align*}
\mathcal{R}(X) &=\left\{\varphi:\N\to\C \,|\, \varphi\text{ defines a radial Schur multiplier on }X\right\},\\
\mathcal{R}_+(X) &=\left\{\varphi:\N\to\C \,|\, \varphi\text{ defines a positive definite radial kernel on }X\right\}.
\end{align*}
It follows that $\mathcal{R}(X)$ is a linear subspace of $\ell_{\infty}(\N)$, and $\mathcal{R}_+(X)$ is a convex cone in $\mathcal{R}(X)$.

Let $2\leq q\leq\infty$ and let $\mathcal{T}_q$ denote the $(q+1)$-homogeneous tree (the choice of this somewhat strange notation will become clearer afterwards). It was proven by Haagerup, Steenstrup and Szwarc \cite{HaaSteSzw} that a function $\varphi:\N\to\C$ belongs to $\mathcal{R}(\mathcal{T}_q)$ if and only if the Hankel matrix
\begin{align}\label{def_H}
H=(\varphi(i+j)-\varphi(i+j+2))_{i,j\in\N}
\end{align}
is of trace class. See \cite[\S 2.4]{Mur} for details on trace-class operators. Observe that the aforementioned characterisation is independent of $q$. In other words, $\mathcal{R}(\mathcal{T}_q)=\mathcal{R}(\mathcal{T}_\infty)$ for all $2\leq q<\infty$. Since $\mathcal{R}_+(\mathcal{T}_q)\subset\mathcal{R}(\mathcal{T}_q)$, the elements of $\mathcal{R}_+(\mathcal{T}_q)$ must at least satisfy this trace-class condition. 

Fix $2\leq q<\infty$ and let $(P_n^{(q)})_{n\in\N}$ be the family of polynomials defined recursively as follows:
\begin{align*}
P_0^{(q)}(x)=1,\quad  P_1^{(q)}(x)=x,
\end{align*}
and
\begin{align}\label{polynom_P_n}
P_{n+1}^{(q)}(x)=\left(1+\tfrac{1}{q}\right)xP_n^{(q)}(x)-\tfrac{1}{q}P_{n-1}^{(q)}(x),\quad\forall n\geq 1.
\end{align}
For $z\in\C$ fixed, the function $n\mapsto P_n^{(q)}(z)$ corresponds to an eigenfunction of the Laplace operator on $\mathcal{T}_q$, with eigenvalue $z$. Arnaud \cite{Arn} proved that a function $\varphi:\N\to\C$ belongs to $\mathcal{R}_+(\mathcal{T}_q)$ if and only if there exists a positive Borel measure $\mu$ on $[-1,1]$ such that
\begin{align*}
\varphi(n)=\int_{-1}^1 P_n^{(q)}(s)\,d\mu(s),
\end{align*}
for all $n\in\N$. See also \cite{Car, CohDeM, FigNeb}. Haagerup and Knudby \cite{HaaKnu} gave a new proof of this fact for odd values of $q$, and extended it to $q=\infty$. In the latter case, the sequence $(P_n^{(\infty)})$ is given by
\begin{align*}
P_n^{(\infty)}(x)=x^n,\quad\forall n\in\N.
\end{align*}
In this paper, we provide a new proof of this characterisation by means of the Haagerup--Steenstrup--Szwarc formula for radial Schur multipliers. We prove the following.

\begin{thm}\label{Thm_trees}
Let $2\leq q\leq\infty$ and let $\varphi:\N\to\C$. The following are equivalent:
\begin{itemize}
\item[a)] The function $\varphi$ belongs to $\mathcal{R}_+(\mathcal{T}_q)$.
\item[b)] The operator $B=(B_{i,j})_{i,j\in\N}$ given by
\begin{align*}
B_{i,j} &=\varphi(i+j)-\varphi(i+j+2),& & (q=\infty)\\
B_{i,j} &=\sum_{k=0}^{\min\{i,j\}}\frac{1}{q^k}\left(\varphi(i+j-2k)-\varphi(i+j-2k+2)\right),& & (q<\infty)
\end{align*}
is of trace class and positive. Moreover, the following limits
\begin{align*}
l_0&=\lim_{n\to\infty}\varphi(2n), & l_1&=\lim_{n\to\infty}\varphi(2n+1)
\end{align*}
 exist, and
\begin{align*}
|l_1|\leq l_0.
\end{align*}
\item[c)] There exists a positive Borel measure on $[-1,1]$ such that
\begin{align*}
\varphi(n)=\int_{-1}^1 P_n^{(q)}(s)\,d\mu(s),
\end{align*}
for all $n\in\N$.
\end{itemize}
\end{thm}

The equivalence (a)$\Leftrightarrow$(b) will follow from the characterisation of $\mathcal{R}(\mathcal{T}_q)$ proven in \cite{HaaSteSzw}. The proof of (b)$\Rightarrow$(c) is inspired by Hamburger's solution to the moment problem \cite{Ham}. Observe that when $q=\infty$, $B$ coincides with the Hankel matrix $H$ defined in \eqref{def_H}.

Since $\mathcal{R}(\mathcal{T}_q)=\mathcal{R}(\mathcal{T}_\infty)$ for all $2\leq q<\infty$, a natural question is what is the relation between $\mathcal{R}_+(\mathcal{T}_q)$ and $\mathcal{R}_+(\mathcal{T}_r)$ for $q\neq r$. Observe that there is a sequence of isometric embeddings
\begin{align*}
\mathcal{T}_2\hookrightarrow\mathcal{T}_3\hookrightarrow\cdots\hookrightarrow\mathcal{T}_{\infty}.
\end{align*}
Such embeddings can be defined as follows. Choose any two vertices $x\in\mathcal{T}_q$, $y\in\mathcal{T}_{q+1}$, and map $x\mapsto y$. Then inductively map neighbours to neighbours. Since the degree of $\mathcal{T}_{q+1}$ is bigger, there is always enough room for this procedure, and since there are no cycles, the resulting map is isometric. This yields a sequence of inclusions
\begin{align}\label{inclusions_RT}
\mathcal{R}_+(\mathcal{T}_2)\supset\mathcal{R}_+(\mathcal{T}_3)\supset\cdots\supset\mathcal{R}_+(\mathcal{T}_\infty).
\end{align}
Moreover, since the condition \eqref{ineq_pd} involves only finitely many points, these isometric embeddings allow one to see that
\begin{align*}
\mathcal{R}_+(\mathcal{T}_\infty)=\bigcap_{q\geq 2}\mathcal{R}_+(\mathcal{T}_{q}).
\end{align*}
As a consequence of Theorem \ref{Thm_trees}, we obtain the following.

\begin{cor}\label{Cor_strct_incl}
For all $2\leq q<\infty$, we have
\begin{align*}
\mathcal{R}_+(\mathcal{T}_{q+1})\subsetneq\mathcal{R}_+(\mathcal{T}_{q}).
\end{align*}
\end{cor}

In the proof of Corollary \ref{Cor_strct_incl}, we construct explicit examples of elements $\varphi\in\mathcal{R}_+(\mathcal{T}_{q})\setminus\mathcal{R}_+(\mathcal{T}_{q+1})$, namely $\varphi(n)=P_n^{(q)}(0)$, where $(P_n^{(q)})$ is the family of polynomials defined in \eqref{polynom_P_n}.

We also study positive definite kernels on products of trees. Given a family of connected graphs $X_1,...,X_N$, the product
\begin{align*}
X=X_1\times\cdots\times X_N
\end{align*}
is the graph whose vertices are $N$-tuples $(x_1,...,x_N)$, with $x_i\in X_i$, and there is an edge between $(x_1,...,x_N)$ and $(y_1,...,y_N)$ if and only if
\begin{align*}
\sum_{i=0}^N d_i(x_i,y_i) = 1.
\end{align*}
Here $d_i$ stands for the edge-path distance on $X_i$. Observe that every $X_i$ can be isometrically embedded into $X$ by simply fixing the other coordinates. We will prove an extension of Theorem \ref{Thm_trees} to products. For $n=(n_1,...,n_N)\in\N^N$ and $m=(m_1,...,m_N)\in\N^N$, we denote by $m\wedge n$ the element of $\N^N$ given by
\begin{align*}
\left(m\wedge n\right)_i=\min\{m_i,n_i\},\quad\forall i\in\{1,...,N\}.
\end{align*}
We will also write $m\leq n$ if $m_i\leq n_i$ for all $i$. Fix $2\leq q_1,...,q_N<\infty$ and set $X=\mathcal{T}_{q_1}\times\cdots\times \mathcal{T}_{q_N}$. Now we consider $N$ families of polynomials $(P_n^{(q_i)})_{n\in\N}$ ($i=1,...,N$) defined as in \eqref{polynom_P_n}.
%\begin{align*}
%P_0^{(q_i)}(x)=1,\quad  P_1^{(q_i)}(x)=x,
%\end{align*}
%\begin{align*}
%P_{n+2}^{(q_i)}(x)=\left(1+\tfrac{1}{q_i}\right)xP_{n+1}^{(q_i)}(x)-\tfrac{1}{q_i}P_{n}^{(q_i)}(x),\quad\forall n\geq 0.
%\end{align*}

\begin{thm}\label{Thm_prods}
Let $\varphi:\N\to\C$, and let $X$ be as above. The following are equivalent:
\begin{itemize}
\item[a)] The function $\varphi$ belongs to $\mathcal{R}_+(X)$.
\item[b)] The operator $B=(B_{n,m})_{m,n\in\N^N}$ on $\ell_2(\N^N)$, given by
\begin{align*}
B_{m,n}=\sum_{l\leq m\wedge n}\frac{1}{q_1^{l_1}\cdots q_N^{l_N}}\sum_{k=0}^N\binom{N}{k}(-1)^k \varphi(|m|+|n|-2|l|+2k),
\end{align*}
is of trace class and positive. Moreover, the following limits
\begin{align*}
l_0&=\lim_{n\to\infty}\varphi(2n), & l_1&=\lim_{n\to\infty}\varphi(2n+1)
\end{align*}
 exist, and
\begin{align*}
|l_1|\leq l_0.
\end{align*}
\item[c)] There exists a positive Borel measure $\mu$ on $[-1,1]^N$ such that
\begin{align}\label{phi=int_P1_PN}
\varphi(n_1+\cdots +n_N)=\int\limits_{[-1,1]^N}P_{n_1}^{(q_1)}(t_1)\cdots P_{n_N}^{(q_N)}(t_N)\, d\mu(t_1,...,t_N),
\end{align}
for all $(n_1,...,n_N)\in\N^N$.
\end{itemize}
\end{thm}

Observe that Theorem \ref{Thm_trees} (for $q<\infty$) is a particular case of Theorem \ref{Thm_prods}. The latter could also be stated in a more general form, allowing some of the $q_i$ to be $\infty$, but this case does not really give new information, as we will explain now.

Let $X$ be a connected graph. For $x,y\in X$, we define the interval between $x$ and $y$ as the set
\begin{align*}
I(x,y)=\{u\in X\, :\, d(x,y)=d(x,u)+d(u,y)\}.
\end{align*}
We say that  $X$ is median if 
\begin{align*}
|I(x,y)\cap I(y,z)\cap I(z,x)|=1,\quad\forall x,y,z\in X.
\end{align*}
The class of median graphs is very large, and it contains trees and finite products of trees. It was proven by Chepoi \cite{Che} that a graph is median if and only if it is the 1-skeleton of a CAT(0) cube complex. The following result is essentially a consequence of the fact that the edge-path distance on such graphs is a conditionally negative kernel (see \cite{NibRee}).

\begin{prop}\label{Prop_median}
Let $X$ be a median graph. Then
\begin{align*}
\mathcal{R}_+(\mathcal{T}_\infty)\subset \mathcal{R}_+(X).
\end{align*}
\end{prop}

Therefore, if a median graph $X$ contains an isometric copy of $\mathcal{T}_\infty$, then $\mathcal{R}_+(X)=\mathcal{R}_+(\mathcal{T}_\infty)$, and we already know what this set looks like, thanks to Theorem \ref{Thm_trees}. In particular,
\begin{align*}
\mathcal{R}_+(\mathcal{T}_\infty)= \mathcal{R}_+(\mathcal{T}_\infty^N),
\end{align*}
for all $N\geq 1$. Here $\mathcal{T}_\infty^N$ stands for the direct product of $N$ copies of $\mathcal{T}_\infty$. So again, this is a very different situation than the one of radial Schur multipliers. Indeed, the following is a consequence of \cite[Theorem A]{Ver} and \cite[Proposition 7.5]{Ver}.

\begin{thm}[\cite{Ver}]
Let $N$ and $q$ be natural numbers with $N\geq 1$ and $2\leq q<\infty$. Then
\begin{align*}
\mathcal{R}(\mathcal{T}_q^N)=\mathcal{R}(\mathcal{T}_\infty^N),
\end{align*}
and
\begin{align*}
\mathcal{R}(\mathcal{T}_\infty^{N+1})\subsetneq\mathcal{R}(\mathcal{T}_\infty^N).
\end{align*}
\end{thm}

Again, since there is an isometric inclusion $\mathcal{T}_q^N\hookrightarrow\mathcal{T}_{q'}^{N'}$ for $q'\geq q$ and $N'\geq N$, we know that
\begin{align}\label{R(T_q^N)sub}
\mathcal{R}_+(\mathcal{T}_{q'}^{N'})\subset\mathcal{R}_+(\mathcal{T}_{q}^{N}).
\end{align}
It seems natural to ask whether some of these inclusions are actually equalities. We do not have a complete answer to this question, but Theorem \ref{Thm_prods} allows us to conclude the following.

\begin{cor}\label{Cor_strct_incl2}
For all $2\leq q<\infty$, we have
\begin{align*}
\mathcal{R}_+(\mathcal{T}_q\times\mathcal{T}_q)\subsetneq\mathcal{R}_+(\mathcal{T}_{q}).
\end{align*}
\end{cor}

We make one final remark concerning this question. The equivalence (a)$\Leftrightarrow$(c) in Theorem \ref{Thm_prods} associates, to each element of $\mathcal{R}_+(X)$, a Borel measure on $[-1,1]^N$, but not conversely. Indeed, the left-hand side of \eqref{phi=int_P1_PN} depends only on the sum $n_1+\cdots +n_N$, and this imposes severe restrictions on the measure $\mu$. To illustrate this fact, suppose that $\varphi:\N\to\C$ satisfies
\begin{align*}
\varphi(n_1+n_2)=\int\limits_{[-1,1]^2}t_1^{n_1}t_2^{n_2}\, d\mu(t_1,t_2),\quad\forall(n_1,n_2)\in\N^2,
\end{align*}
for some positive measure $\mu$ on $[-1,1]^2$. Then, by computing the integral of $t_1^{n_1}t_2^{n_2}(t_1-t_2)^2$, one can show that $\mu$ must be supported on the diagonal of $[-1,1]^2$. This is actually another way of proving that $\mathcal{R}_+(\mathcal{T}_\infty^2)=\mathcal{R}_+(\mathcal{T}_\infty)$. This suggests that a more thorough analysis of the identity \eqref{phi=int_P1_PN} could lead to a better understanding of the set $\mathcal{R}_+(\mathcal{T}_{q}^{N})$ and the inclusions \eqref{R(T_q^N)sub}.

\subsection{Organisation of the paper}
In Section \ref{Sect_prelim}, we state some general results concerning completely positive multipliers and trace-class operators. Section \ref{Sect_prod_trees} is devoted to the equivalence (a)$\Leftrightarrow$(b) in both Theorems \ref{Thm_trees} and \ref{Thm_prods}. In Section \ref{Sect_Ham}, we study a variant of the Hamburger moment problem and prove the implication (b)$\Rightarrow$(c). We complete the proof of the main results in Section \ref{Sect_proof_main}. Corollaries \ref{Cor_strct_incl} and \ref{Cor_strct_incl2} are proven in Section \ref{Sect_concr_ex}. Finally, in Section \ref{Sect_median}, we discuss median graphs and prove Proposition \ref{Prop_median}.

\section{Preliminaries}\label{Sect_prelim}

We present here some general facts on completely positive multipliers and positive trace-class operators, which will be crucial in our proofs.

\subsection{Completely positive multipliers}

Let $A$ and $B$ be C${}^*$-algebras. A linear map $u:A\to B$ is said to be completely bounded if
\begin{align*}
\|u\|_{cb}=\sup_{n\geq 1}\left\|id_n\otimes u:M_n(A)\to M_n(B)\right\| <\infty.
\end{align*}
It is completely positive if $id_n\otimes u$ is positive for each $n$. The following fact is essential for our purposes. See e.g. \cite[Corollary 1.8]{Pis} for a proof.

\begin{prop}\label{Prop_car_pd_maps}
Let $A$ and $B$ be unital C${}^*$-algebras and let $u:A\to B$ be a unital linear map. Then $u$ is completely positive if and only if
\begin{align*}
\|u\|_{cb}=1.
\end{align*}
\end{prop}

Every Schur multiplier $\psi:X\times X\to\C$ defines a completely bounded map $M_\psi:\mathcal{B}(X)\to\mathcal{B}(X)$, and the cb norm coincides with the usual operator norm. See \cite[Theorem 5.1]{Pis2} for a proof. We shall write
\begin{align*}
\|\psi\|_{cb}=\|M_\psi\|_{cb}=\|M_\psi\|.
\end{align*}

The following characterisation of positive definite radial kernels will be one of our main tools. It is a direct consequence of Propositions \ref{Prop_cp_pd} and \ref{Prop_car_pd_maps}.

\begin{prop}\label{Prop_norm_pd}
Let $X$ be a connected graph and let $\psi:X\times X\to\C$ be a radial Schur multiplier defined by a function $\varphi:\N\to\C$. Then $\psi$ is positive definite if and only if
\begin{align*}
\|\psi\|_{cb}=\varphi(0).
\end{align*}
\end{prop}

\subsection{Positive trace-class operators}

Let $\mathcal{H}$ be a separable Hilbert space. We denote by $S_1(\mathcal{H})$ the space of trace-class operators on $\mathcal{H}$, and by $S_1(\mathcal{H})_+$ the subset of positive elements of $S_1(\mathcal{H})$. See \cite[\S 2.4]{Mur} for details on trace-class operators. The following fact is well known. We include its proof for the sake of completeness.

\begin{lem}\label{Lem_pos_Tr}
Let $A\in S_1(\mathcal{H})$. Then $A$ is positive if and only if Tr$(A)=\|A\|_{S_1}$.
\end{lem}
\begin{proof}
If $A$ is positive, then $\|A\|_{S_1}=$ Tr$(A)$ by the definition of the $S_1$-norm. Suppose now that $\|A\|_{S_1}=$ Tr$(A)$. By the singular value decomposition of compact operators (see e.g. \cite[Theorem 1.84]{Att}), $A$ can be written as
\begin{align*}
A=\sum_{n\in\N}\lambda_n u_n\odot v_n,
\end{align*}
where $(\lambda_n)$ is a sequence of non-negative numbers converging to $0$, $(u_n)$ and $(v_n)$ are two orthonormal families in $\mathcal{H}$, and $u_n\odot v_n$ is the rank 1 operator given by $(u_n\odot v_n)f=\langle f,v_n\rangle u_n$. Moreover, if $A=U|A|$ is the polar decomposition of $A$, then
\begin{align*}
|A|=\sum_{n\in\N}\lambda_n v_n\odot v_n,
\end{align*}
and $u_n=Uv_n$. We have
\begin{align*}
\|A\|_{S_1}=\text{Tr}(|A|)=\sum\lambda_n,
\end{align*}
and
\begin{align*}
\text{Tr}(A) = \sum_{n\in\N}\langle Av_n,v_n\rangle = \sum_{n\in\N}\langle \lambda_nu_n,v_n\rangle 
= \sum_{n\in\N}\lambda_n\langle Uv_n,v_n\rangle.
\end{align*}
Hence
\begin{align*}
\sum_{n\in\N}\lambda_n=\sum_{n\in\N}\lambda_n\langle Uv_n,v_n\rangle.
\end{align*}
Since $\lambda_n\geq 0$ and $|\langle Uv_n,v_n\rangle|\leq 1$, we must have $\langle Uv_n,v_n\rangle=1$ for all $n$, which implies that $Uv_n=v_n$. Therefore $u_n=v_n$ and $A=|A|$.
\end{proof}

\section{Products of trees and multi-radial kernels}\label{Sect_prod_trees}

In this section we deal with the equivalence (a)$\Leftrightarrow$(b) in Theorems \ref{Thm_trees} and \ref{Thm_prods}. We will work in the slightly more general context of multi-radial kernels. The main reason is that this will make computations on products a little easier, but it may also be of independent interest.

\subsection{Homogeneous trees}
We begin by recalling some facts and definitions from \cite{HaaSteSzw}. Let $\mathcal{T}_q$ be the $(q+1)$-homogeneous tree ($2\leq q<\infty$). We fix an infinite geodesic ray $\omega_0:\N\to \mathcal{T}_q$. For each $x\in \mathcal{T}_q$ there exists a unique geodesic ray $\omega_x:\N\to \mathcal{T}_q$ which eventually flows with $\omega_0$, meaning that $|\omega_x\Delta\omega_0|<\infty$, where we view $\omega_0$ and $\omega_x$ as subsets of $\mathcal{T}_q$. Using these rays, we can construct a family of vectors $(\delta_x')_{x\in \mathcal{T}_q}$ in $\ell_2(\mathcal{T}_q)$ such that
\begin{align*}
\langle\delta_x',\delta_y'\rangle = \begin{cases}
1 & \text{if } x=y,\\
-\frac{1}{q-1} & \text{if } x\neq y \text{ and } \omega_x(1)=\omega_y(1),\\
0 & \text{otherwise.}
\end{cases}
\end{align*}
See the discussion in \cite{HaaSteSzw} after Remark 2.4. Now let $S$ denote the forward shift operator on $\ell_2(\N)$, and define
\begin{align}\label{def_S_mn}
S_{m,n} = \left(1-\frac{1}{q}\right)^{-1}\left(S^m(S^*)^n-\frac{1}{q} S^*S^m(S^*)^nS\right),
\end{align}
for all $m,n\in\N$.

\begin{lem}{\cite[Lemma 2.5]{HaaSteSzw}}\label{Lem_S_mn}
Let $x,y\in\mathcal{T}_q$. Let $m$ and $n$ be the smallest numbers such that $\omega_x(m)\in\omega_y$ and $\omega_y(n)\in\omega_x$ (this implies $\omega_x(m)=\omega_y(n)$ and $d(x,y)=m+n$). Then
\begin{align*}
(S_{m,n})_{ij}=\left\langle\delta_{\omega_y(j)}',\delta_{\omega_x(i)}'\right\rangle,
\end{align*}
for all $i,j\in\N$.
\end{lem}

\subsection{Multi-radial kernels}

Let $X=X_1\times\cdots\times X_N$ be a product of $N$ connected graphs. We say that $\phi:X\times X\to\C$ is a multi-radial function if there exists $\tilde{\phi}:\N^N\to\C$ such that
\begin{align*}
\phi(x,y)=\tilde{\phi}(d(x_1,y_1),...,d(x_N,y_N)),\quad\forall x,y\in X,
\end{align*}
where $x=(x_i)_{i=1}^N$, $y=(y_i)_{i=1}^N$, and $d_i$ is the edge-path distance on $X_i$. It readily follows that a radial function $\phi:X\times X\to\C$, given by $\phi(x,y)=\varphi(d(x,y))$, is multi-radial with
\begin{align*}
\tilde{\phi}(n_1,...,n_N)=\varphi(n_1+\cdots +n_N),\quad \forall (n_1,...,n_N)\in\N^N.
\end{align*}
We recall some facts and notations from \cite[\S 2]{Ver}. Put $[N]=\{1,...,N\}$. For each $I\subset[N]$, we define $\chi^I\in\{0,1\}^N$ by
\begin{align}\label{chi_I}
\chi^I_i=\begin{cases} 1 & \text{if } i\in I \\
0 & \text{if } i\notin I.
\end{cases}
\end{align}
Consider the space $\ell_2(\N^N)=\ell_2(\N)\otimes\cdots\otimes\ell_2(\N)$ and denote by $S_i$ the forward shift operator on the $i$-th coordinate. Namely,
\begin{align*}
S_1\left(\delta_{n_1}\otimes\delta_{n_2}\otimes\cdots\otimes\delta_{n_N}\right) = \left(\delta_{n_1+1}\otimes\delta_{n_2}\otimes\cdots\otimes\delta_{n_N}\right),
\end{align*}
and so forth. For each $i\in\{1,...,N\}$, define $S_{i,m,n}\in\mathcal{B}(\ell_2(\N^N))$ like in \eqref{def_S_mn}, and put
\begin{align}\label{def_S(m,n)}
S(m,n)=\prod_{i=1}^N S_{i,m_i,n_i},\quad\forall m,n\in\N^N.
\end{align}

\subsection{Products of homogeneous trees}\label{Subsect_prod_trees}
A characterisation of multi-radial multipliers on products of trees was given in \cite{Ver}, together with some estimates on the norms. We shall revisit the proof of this result and obtain an exact formula for the norm in the homogeneous case. This will generalise \cite[Theorem 1.2]{HaaSteSzw}. For details on Schur multipliers, see \cite[Chapter 5]{Pis2}.

From now on, fix $2\leq q_1,...,q_N<\infty$ and let $\mathcal{T}_{q_i}$ denote the $(q_i+1)$-homogeneous tree ($i=1,...,N$). We will focus on the graph
\begin{align*}
X=\mathcal{T}_{q_1}\times\cdots\times\mathcal{T}_{q_N}.
\end{align*}
Consider, for each $i=1,...,N$, the operator $\tau_i:\mathcal{B}(\ell_2(\N^N))\to\mathcal{B}(\ell_2(\N^N))$ defined by
\begin{align*}
\tau_i(T)=S_iTS_i^*.
\end{align*}
Since $\tau_i$ is isometric on $S_1(\ell_2(\N^N))$ (see the proof of \cite[Lemma 2.14]{Ver}), the formula
\begin{align*}
\prod_{i=1}^N\left(1-\frac{1}{q_i}\right)^{-1}\left(I-\frac{\tau_i}{q_i}\right)
\end{align*}
defines an isomorphism of $S_1(\ell_2(\N^N))$. Observe that this is a product of a commuting family of operators, hence there is no ambiguity in this definition.

\begin{lem}\label{Lem_norm_prod}
Let $\tilde{\phi}:\N^N \to\C$ be a function such that the limits
\begin{align*}
l_0&=\lim_{\substack{|n|\to\infty \\ |n|\text{ even}}}\tilde{\phi}(n) & l_1&=\lim_{\substack{|n|\to\infty \\ |n|\text{ odd}}}\tilde{\phi}(n)
\end{align*}
exist. Then $\tilde{\phi}$ defines a multi-radial Schur multiplier on $X$ if and only if the operator $T=(T_{n,m})_{m,n\in\N^N}$ given by
\begin{align}\label{def_T}
T_{m,n}=\sum_{I\subset[N]}(-1)^{|I|} \tilde{\phi}(m+n+2\chi^I),\quad\forall m,n\in\N^N
\end{align}
is an element of $\in S_1(\ell_2(\N^N))$. Moreover, in that case, the associated multiplier $\phi$ satisfies
\begin{align}\label{exact_norm_phi}
\|\phi\|_{cb}=\left\| \prod_{i=1}^N\left(1-\frac{1}{q_i}\right)\left(I-\frac{\tau_i}{q_i}\right)^{-1}T\right\|_{S_1} + |c_+|+|c_-|,
\end{align}
where $c_\pm=l_0\pm l_1$.
\end{lem}
\begin{proof}
The characterisation of such multipliers was already proven in \cite[Proposition 2.1]{Ver}. Hence we only need to show that the identity \eqref{exact_norm_phi} holds. Observe that one of the inequalities is somehow implicit in \cite{Ver}. By \cite[Lemma 2.13]{Ver}, there exists a trace-class operator $T'$ such that
\begin{align}\label{phi(m+n)=Tr}
\tilde{\phi}(m+n)=c_++(-1)^{|m|+|n|}c_-+\text{Tr}\left(S(m,n)T'\right),\quad\forall m,n\in\N,
\end{align}
and
\begin{align*}
\left\| T' \right\|_{S_1} + |c_+|+|c_-| \leq \|\phi\|_{cb}.
\end{align*}
Then it is shown in (the proof of) \cite[Lemma 2.8]{Ver} that
\begin{align}\label{def_T'}
T'=\prod_{i=1}^N\left(1-\frac{1}{q_i}\right)\left(I-\frac{\tau_i}{q_i}\right)^{-1}T.
\end{align}
In order to prove the other inequality, we shall use the characterisation of Schur multipliers as scalar products on a Hilbert space (see \cite[Theorem 5.1]{Pis2}). Let $T'$ be as in \eqref{def_T'}. Since $T'$ is of trace class, there exist sequences $\left(\xi^{(k)}\right)_{k\geq 0}$, $\left(\eta^{(k)}\right)_{k\geq 0}$ in $\ell_2(\N^N)$ such that
\begin{align*}
T'=\sum_{k\geq 0}\xi^{(k)}\odot\eta^{(k)},
\end{align*}
and
\begin{align*}
\|T'\|_{S_1}=\sum_{k\geq 0}\|\xi^{(k)}\|_2\|\eta^{(k)}\|_2.
\end{align*}
See the discussion right after \cite[Corollary 2.7]{HaaSteSzw}. This implies that
\begin{align*}
T_{m,n}'=\sum_{k\geq 0} \xi^{(k)}_m \overline{\eta^{(k)}_n},\quad\forall m,n\in\N^N.
\end{align*}
Now, for each $k\in\N$, define functions $P_k,Q_k:X\to\ell_2(\N^N)$ by
\begin{align*}
P_k(x)&=\sum_{l_1,...,l_N\geq 0}\xi^{(k)}_{(l_1,...,l_N)} \delta_{\omega_{x_1}(l_1)}'\otimes\cdots\otimes\delta_{\omega_{x_N}(l_N)}',\\
Q_k(y)&=\sum_{j_1,...,j_N\geq 0}\eta^{(k)}_{(j_1,...,j_N)} \delta_{\omega_{y_1}(j_1)}'\otimes\cdots\otimes\delta_{\omega_{y_N}(j_N)}',
\end{align*}
for all $x=(x_1,...,x_N)\in X$ and $y=(y_1,...,y_N)\in X$. Observe that
\begin{align*}
\|P_k(x)\|^2=\sum_{l_1,...,l_N\geq 0}\left\|\xi^{(k)}_{(l_1,...,l_N)}\right\|^2=\|\xi^{(k)}\|^2,\quad\forall x\in X.
\end{align*}
Similarly, $\|Q_k(y)\|=\|\eta^{(k)}\|$ for all $y\in X$. Hence
\begin{align*}
\sum_{k\geq 0}\|P_k\|_\infty\|Q_k\|_\infty = \sum_{k\geq 0}\|\xi^{(k)}\|\|\eta^{(k)}\| = \|T'\|_{S_1}.
\end{align*}
Now fix $x,y\in X$ and observe that
\begin{align}
&\sum_{k\geq 0}\langle P_k(x),Q_k(y)\rangle\notag\\
 &= \sum_{k\geq 0} \sum_{\substack{l_1,...,l_N=0\\ j_1,...,j_N=0}}^{\infty} \left(\prod_{i=1}^N\left\langle\delta_{\omega_{x_i}(l_i)},\delta_{\omega_{y_i}(j_i)}\right\rangle\right)\xi^{(k)}_{(l_1,...,l_N)} \overline{\eta^{(k)}_{(j_1,...,j_N)}}\notag\\
&= \sum_{\substack{l_1,...,l_N=0\\ j_1,...,j_N=0}}^{\infty} \left(\prod_{i=1}^N\left\langle\delta_{\omega_{x_i}(l_i)},\delta_{\omega_{y_i}(j_i)}\right\rangle\right)T_{(l_1,...,l_N),(j_1,...,j_N)}'.\label{Sum_PkQk}
\end{align}
Then, by Lemma \ref{Lem_S_mn}, there exist $m_1,...,m_N$ and $n_1,...,n_N$ in $\N$ such that $d(x_i,y_i)=m_i+n_i$ and
\begin{align*}
\left\langle\delta_{\omega_x(l_i)}',\delta_{\omega_y(j_i)}'\right\rangle=(S_{i,m_i,n_i})_{j_il_i},
\end{align*}
for all $l_i,j_i\in\N$ ($i=1,...,N$). Hence, using \eqref{def_S(m,n)}, the identity \eqref{Sum_PkQk} can be rewritten as
\begin{align*}
\sum_{k\geq 0}\langle P_k(x),Q_k(y)\rangle &= \sum_{l,j\in\N^N} S(m,n)_{j,l}T_{l,j}'\\
&= \text{Tr}\left(S(m,n)T'\right),
\end{align*}
which, by \eqref{phi(m+n)=Tr}, gives
\begin{align*}
\phi(x,y)=\tilde{\phi}(m+n)=c_++(-1)^{|m|+|n|}c_-+\sum_{k\geq 0}\langle P_k(x),Q_k(y)\rangle.
\end{align*}
Observe that the function $(m,n)\mapsto(-1)^{|m|+|n|}$ defines a Schur multiplier on $X$ of norm $1$ (see e.g. \cite[Lemma 2.3]{Ver}). Therefore
\begin{align*}
\|\phi\|_{cb} &\leq |c_+|+|c_-|+\sum_{k\geq 0}\|P_k\|_\infty\|Q_k\|_\infty\\
&=|c_+|+|c_-|+ \|T'\|_{S_1}.
\end{align*}
\end{proof}

Now we will use this lemma to characterise positive definite multi-radial kernels on $X$. Observe that a more general form of Proposition \ref{Prop_norm_pd} still holds in this context. Namely, $\phi$ is a positive definite multi-radial kernel if and only if
\begin{align*}
\|\phi\|_{cb}=\tilde{\phi}(0,...,0).
\end{align*}

\begin{cor}\label{Cor_multirad_pd}
Let $\tilde{\phi}:\N^N \to\C$ be a function such that the limits
\begin{align}\label{limits_l0_l1}
l_0&=\lim_{\substack{|n|\to\infty \\ |n|\text{ even}}}\tilde{\phi}(n) & l_1&=\lim_{\substack{|n|\to\infty \\ |n|\text{ odd}}}\tilde{\phi}(n)
\end{align}
exist, and let $T$ be as in \eqref{def_T}. Then $\tilde{\phi}$ defines a positive definite multi-radial kernel on $X$ if and only if the operator
\begin{align}\label{def_T'_1}
T'=\prod_{i=1}^N\left(1-\frac{1}{q_i}\right)\left(I-\frac{\tau_i}{q_i}\right)^{-1}T
\end{align}
is an element of $\in S_1(\ell_2(\N^N))_+$, and
\begin{align*}
|l_1|\leq l_0.
\end{align*}
\end{cor}
\begin{proof}
We know from Lemma \ref{Lem_norm_prod} that $\tilde{\phi}$ defines a multi-radial Schur multiplier $\phi$ on $X$ if and only if $T$ (equivalently $T'$) is of trace class. And in that case
\begin{align*}
\|\phi\|_{cb}=\|T'\|_{S_1}+|c_+|+|c_-|,
\end{align*}
where $c_\pm=l_0\pm l_1$. On the other hand, we know that $\phi$ is positive definite if and only if $\|\phi\|_{cb}=\tilde{\phi}(0,...,0)$, or equivalently,
\begin{align*}
\tilde{\phi}(0,...,0) = \|T'\|_{S_1} + |c_+| + |c_-|.
\end{align*}
Moreover, by \eqref{phi(m+n)=Tr},
\begin{align*}
\tilde{\phi}(0,...,0) =c_++c_-+\text{Tr}(T').
\end{align*}
This means that $\phi$ is positive definite if and only if
\begin{align*}
\text{Tr}(T')+c_++c_-= \|T'\|_{S_1} + |c_+| + |c_-|.
\end{align*}
By Lemma \ref{Lem_pos_Tr}, this is equivalent to $T' \in S_1(\ell_2(\N^N))_+$ and $c_+,c_-\geq 0$. Finally observe that the latter condition can be rewritten as $|l_1|\leq l_0$.
\end{proof}

Now we go back to radial kernels. Recall that for $m,n\in\N^N$, we denote by $m\wedge n$ the element of $\N^N$ given by
\begin{align*}
\left(m\wedge n\right)_i=\min\{m_i,n_i\},\quad\forall i\in\{1,...,N\}.
\end{align*}
By $m\leq n$ we mean $m_i\leq n_i$ for all $i$.

\begin{cor}\label{Cor_pd_prod}
A function $\varphi:\N \to\C$ belongs to $\mathcal{R}_+(X)$ if and only if the following two conditions hold
\begin{itemize}
\item[a)] The operator $B=(B_{n,m})_{m,n\in\N^N}$ given by
\begin{align*}
B_{m,n}=\sum_{l\leq m\wedge n}\frac{1}{q_1^{l_1}\cdots q_N^{l_N}}\sum_{k=0}^N\binom{N}{k}(-1)^k \varphi(|m|+|n|-2|l|+2k),
\end{align*}
is of trace class and positive.
\item[b)] The following limits
\begin{align*}
l_0&=\lim_{n\to\infty}\varphi(2n), & l_1&=\lim_{n\to\infty}\varphi(2n+1)
\end{align*}
 exist, and
\begin{align*}
|l_1|\leq l_0.
\end{align*}
\end{itemize}
\end{cor}
\begin{proof}
We will apply Corollary \ref{Cor_multirad_pd} to the function $\tilde{\phi}:\N^N\to\C$ given by
\begin{align*}
\tilde{\phi}(n)=\varphi(|n|),\quad\forall n\in\N^N.
\end{align*}
For this purpose, we need to show that the limits \eqref{limits_l0_l1} exist. If we assume (b), then this is part of the hypotheses. If we suppose that $\varphi\in\mathcal{R}_+(X)$, then by restriction, $\varphi\in\mathcal{R}_+(\mathcal{T}_{q_1})$, and the limits $l_0,l_1$ exist by \cite[Theorem 1.2]{HaaSteSzw}. So $\tilde{\phi}$ satisfies the hypotheses of Corollary \ref{Cor_multirad_pd}. We conclude that $\varphi\in\mathcal{R}_+(X)$ if and only if $T'$ belongs to $S_1(\ell_2(\N^N))_+$ and $|l_1|\leq l_0$. Observe that
\begin{align*}
T_{m,n}&=\sum_{I\subset[N]}(-1)^{|I|} \tilde{\phi}(m+n+2\chi^I)\\
&= \sum_{I\subset[N]}(-1)^{|I|} \varphi(|m|+|n|+2|I|)\\
&= \sum_{k=0}^N\binom{N}{k}(-1)^k\varphi(|m|+|n|+2k),
\end{align*}
for all $m,n\in\N^N$. Moreover,
\begin{align*}
T'=\alpha\sum_{l\in\N^N}\frac{1}{q_1^{l_1}\cdots q_N^{l_N}}S_1^{l_1}\cdots S_N^{l_N}T(S_N^*)^{l_N}\cdots(S_1^*)^{l_1},
\end{align*}
where $\alpha=\prod_{i=1}^N\left(1-\frac{1}{q_i}\right)$. Hence
\begin{align*}
T_{m,n}' &=\alpha\sum_{l\in\N^N}\frac{1}{q_1^{l_1}\cdots q_N^{l_N}}\left\langle T(S_N^*)^{l_N}\cdots(S_1^*)^{l_1}\delta_n, (S_N^*)^{l_N}\cdots(S_1^*)^{l_1}\delta_m\right\rangle\\
&=\alpha\sum_{l\leq m\wedge n}\frac{1}{q_1^{l_1}\cdots q_N^{l_N}}\left\langle T\delta_{n-l}, \delta_{m-l}\right\rangle\\
&=\alpha\sum_{l\leq m\wedge n}\frac{1}{q_1^{l_1}\cdots q_N^{l_N}}\sum_{k=0}^N\binom{N}{k}(-1)^k\varphi(|m|+|n|-2|l|+2k),
\end{align*}
for all $m,n\in\N^N$. This shows that $T'=\alpha B$, which finishes the proof.
\end{proof}

\subsection{The homogeneous tree of infinite degree $\mathcal{T}_{\infty}$}

Corollary \ref{Cor_multirad_pd} deals with homogeneous trees of finite degree. As we discussed in the introduction, an important role is played by the homogeneous tree of infinite degree $\mathcal{T}_{\infty}$. Recall that $\mathcal{R}_+(\mathcal{T}_\infty)$ stands for the set of functions on $\N$ defining positive definite radial kernels on $\mathcal{T}_\infty$.

\begin{lem}\label{Lem_pd_T_infty}
Let $\varphi:\N \to\C$, and let $H$ be the Hankel operator given by
\begin{align*}
H_{i,j}=\varphi(i+j)-\varphi(i+j+2).
\end{align*}
Then $\varphi$ belongs to $\mathcal{R}_+(\mathcal{T}_\infty)$ if and only if $H$ belongs to $S_1(\ell_2(\N))_+$ and
\begin{align*}
\left|\lim_{n\to\infty}\varphi(2n+1)\right|\leq \lim_{n\to\infty}\varphi(2n).
\end{align*}
In particular, those limits exist.
\end{lem}
\begin{proof}
It was proven in \cite[Theorem 1.2]{HaaSteSzw} that $\varphi$ belongs to $\mathcal{R}(\mathcal{T}_\infty)$ if and only if $H$ belongs to $S_1(\ell_2(\N))$, and in that case the associated multiplier $\psi$ satisfies
\begin{align*}
\|\psi\|_{cb}=\|H\|_{S_1}+|c_+|+|c_-|,
\end{align*}
where
\begin{align*}
c_\pm=\frac{1}{2}\lim_{n\to\infty}\varphi(2n)\pm\frac{1}{2}\lim_{n\to\infty}\varphi(2n+1).
\end{align*}
Since Tr$(H)=\varphi(0)-c_+-c_-$, the same arguments as in the proof of Corollary \ref{Cor_multirad_pd} show that $\varphi\in\mathcal{R}_+(\mathcal{T}_\infty)$ if and only if
\begin{align*}
\text{Tr}(H)+c_++c_-=\|H\|_{S_1}+|c_+|+|c_-|,
\end{align*}
and this is equivalent to $H\in S_1(\ell_2(\N))_+$ and $c_+,c_-\geq 0$, which proves the result.
\end{proof}

\section{A variation of the Hamburger moment problem}\label{Sect_Ham}

This section is devoted to the implication (b)$\Rightarrow$(c) in Theorems \ref{Thm_trees} and \ref{Thm_prods}, whose proof is inspired by the Hamburger moment problem.

Given a sequence $(a_n)$ of real numbers, the moment problem asks whether there exists a positive measure $\mu$ on $\R$ such that
\begin{align*}
a_n=\int_{-\infty}^\infty s^n\,d\mu(s),
\end{align*}
for all $n\in\N$. Hamburger proved \cite{Ham} that this happens if and only if the Hankel matrix $H=(a_{i+j})_{i,j\in\N}$ is positive definite. Observe that one of the implications is quite simple. If we assume that $(a_n)$ is the sequence of moments of a positive measure $\mu$, then for every finite sequence $x=(x_i)_{i=0}^m$ in $\C$,
\begin{align*}
\sum_{i,j=0}^m H_{ij}\overline{x_i}x_j &= \sum_{i,j=0}^m a_{i+j}\overline{x_i}x_j\\
&=\int_{-\infty}^\infty\sum_{i,j=0}^m s^{i+j}\overline{x_i}x_j\,d\mu(s)\\
&=\int_{-\infty}^\infty\left|\sum_{i=0}^m s^{i}x_i\right|^2\,d\mu(s),
\end{align*}
which proves that $H$ is positive definite. The converse is more delicate. We refer the reader to \cite[\S 1.7]{Pel} for a full proof of it, which relies on the spectral theory of unbounded self-adjoint operators on Hilbert spaces. Since we deal only with trace-class operators here, our proofs will become simpler.

\subsection{The moment problem for trace-class operators}

We begin with the proof of (b)$\Rightarrow$(c) in Theorem \ref{Thm_trees} for $q=\infty$. The properties of the measures in Hamburger's theorem for bounded, compact and trace-class Hankel operators are well known. See \cite[\S 1.7]{Pel} and \cite{Pow} for more information. Nevertheless, we shall give a complete proof of the construction of the measure in our setting, as it will illustrate the main ideas that we will generalise for finite values of $q$.

\begin{lem}\label{Lem_moment_tc}
Let $\varphi:\N\to\C$ be a function such that the Hankel operator
\begin{align*}
H=\left(\varphi(i+j)-\varphi(i+j+2)\right)_{i,j\in\N}
\end{align*}
belongs to $S_1(\ell_2(\N))_+$, and
\begin{align}\label{ineq_limits}
\left|\lim_{n\to\infty}\varphi(2n+1)\right|\leq \lim_{n\to\infty}\varphi(2n).
\end{align}
Then there exists a positive Borel measure $\mu$ on $[-1,1]$ such that
\begin{align*}
\varphi(n)=\int_{-1}^1 s^n\,d\mu(s),
\end{align*}
for all $n\in\N$.
\end{lem}

\begin{rmk}
The existence of the limits in \eqref{ineq_limits} is a consequence of the fact that $H$ belongs to $S_1(\ell_2(\N))$ (see \cite[Theorem 1.2]{HaaSteSzw}). The real hypothesis here is that the inequality \eqref{ineq_limits} holds.
\end{rmk}

\begin{proof}[Proof of Lemma \ref{Lem_moment_tc}]
Let $c_{00}(\N)$ be the space of finitely supported sequences in $\N$. Since $H$ is positive, we can define a scalar product on $c_{00}(\N)$ by
\begin{align*}
\langle f,g\rangle_{\mathcal{H}}=\langle Hf,g\rangle_{\ell_2}.
\end{align*}
Let $N$ be the subspace consisting of all $f\in c_{00}(\N)$ such that $\langle f,f\rangle_{\mathcal{H}}=0$ (it is possible that $N=\{0\}$). We define $\mathcal{H}$ as the completion of $c_{00}(\N)/N$ for the scalar product $\langle \,\cdot\,,\,\cdot\,\rangle_{\mathcal{H}}$. Let $S$ be the forward shift operator on $c_{00}(\N)$. Since $H$ is a Hankel matrix, we have
\begin{align*}
\langle S f,g\rangle_{\mathcal{H}} &= \langle HSf,g\rangle_{\ell_2}\\
&= \langle Hf,Sg\rangle_{\ell_2}\\
&=\langle f,Sg\rangle_{\mathcal{H}},
\end{align*}
for all $f,g\in c_{00}(\N)$. This tells us two things. It shows that $S$ is symmetric, and that it passes to the quotient. Indeed, by the Cauchy--Schwarz inequality,
\begin{align*}
\langle Sf, Sf\rangle_{\mathcal{H}} &= \langle f, S^2f\rangle_{\mathcal{H}}\\
&\leq \langle f, f\rangle_{\mathcal{H}}^{\frac{1}{2}}\langle S^2f, S^2f\rangle_{\mathcal{H}}^{\frac{1}{2}},
\end{align*}
for all $f\in c_{00}(\N)$. We shall prove that it extends to a bounded self-adjoint operator on $\mathcal{H}$. For this purpose, fix $f\in c_{00}(\N)$ with $\|f\|_{\mathcal{H}}=1$ and define the following sequence,
\begin{align*}
a_n=\|S^{n}f\|_{\mathcal{H}}^2=\langle S^{2n}f,f\rangle_{\mathcal{H}},\quad\forall n\in\N.
\end{align*}
Again, by the Cauchy--Schwarz inequality,
\begin{align*}
a_n\leq \|S^{2n}f\|_{\mathcal{H}}\|f\|_{\mathcal{H}} = (a_{2n})^{\frac{1}{2}}.
\end{align*}
So by induction,
\begin{align}\label{a_1_leq_a_n}
\|S f\|_{\mathcal{H}}^2=a_1\leq(a_{2^n})^{\frac{1}{2^n}},\quad\forall n\in\N.
\end{align}
Now observe that
\begin{align*}
a_n &=\langle S^{2n}f,Hf\rangle_{\ell_2}\\
&\leq \|S\|_{\mathcal{B}(\ell_2)}^{2n}\|H\|_{\mathcal{B}(\ell_2)}\|f\|_{\ell_2}^2\\
&\leq \|H\|_{\mathcal{B}(\ell_2)}\|f\|_{\ell_2}^2.
\end{align*}
Thus the sequence $(a_n)$ is bounded, which by \eqref{a_1_leq_a_n} implies that
\begin{align*}
\|S f\|_{\mathcal{H}}^2\leq\liminf_{n\to\infty}(a_{2^n})^{\frac{1}{2^n}}\leq 1.
\end{align*}
This shows that $S$ extends to a bounded self-adjoint operator $\Psi$ on $\mathcal{H}$ with $\|\Psi\|_{\mathcal{B}(\mathcal{H})}\leq 1$. Let $E_\Psi$ denote its spectral measure (for a very nice introduction to von Neumann's spectral theorem, we refer the reader to Attal's lecture notes \cite{Att}, which are available online). We can define a positive measure $\mu$ on $[-1,1]$ by
\begin{align*}
\mu(A)=\langle E_\Psi(A)\delta_0,\delta_0\rangle_{\mathcal{H}},
\end{align*}
for every Borel set $A\subset[-1,1]$. Hence
\begin{align*}
\int_{-1}^1 s^n\,d\mu(s) &= \langle \Psi^n\delta_0,\delta_0\rangle_{\mathcal{H}}\\
&= \langle HS^n\delta_0,\delta_0\rangle_{\ell_2}\\
&= \varphi(n)-\varphi(n+2),
\end{align*}
for all $n\in\N$. On the other hand, since $H$ is of trace class, we know that
\begin{align*}
\sum_{n\geq 0}\left(\varphi(2n)-\varphi(2n+2)\right) <\infty,
\end{align*}
and this implies that $\mu(\{-1,1\})=0$ and
\begin{align*}
\int\limits_{(-1,1)} \frac{1}{1-s^2}\,d\mu(s) = \int\limits_{(-1,1)} \sum_{n\geq 0} s^{2n}\,d\mu(s) < \infty.
\end{align*}
Let $\nu$ be the finite measure on $(-1,1)$ given by
\begin{align*}
d\nu(s)=(1-s^2)^{-1}d\mu(s).
\end{align*}
Then
\begin{align*}
\varphi(n)-\varphi(n+2)=\int\limits_{(-1,1)} s^n\,d\nu(s) - \int\limits_{(-1,1)} s^{n+2}\,d\nu(s).
\end{align*}
Recall the notation
\begin{align*}
c_\pm=\frac{1}{2}\lim_{n\to\infty}\varphi(2n)\pm \frac{1}{2}\lim_{n\to\infty}\varphi(2n+1).
\end{align*}
By the dominated convergence theorem, for all $n\in\N$,
\begin{align*}
\varphi(n)-(c_++(-1)^nc_-) &=\sum_{j\geq 0}\varphi(n+2j)-\varphi(n+2j+2)\\
&=\int\limits_{(-1,1)} s^n\,d\nu(s) - \lim_{j\to\infty}\int\limits_{(-1,1)} s^{n+2j+2}\,d\nu(s)\\
&=\int\limits_{(-1,1)} s^n\,d\nu(s).
\end{align*}
Since by hypothesis $c_+,c_-\geq 0$, we can define a new positive measure, on the closed interval $[-1,1]$ now, by
\begin{align*}
\tilde{\nu}=c_+\delta_{1}+c_-\delta_{-1}+\nu.
\end{align*}
This measure satisfies
\begin{align*}
\int\limits_{[-1,1]} s^n\,d\tilde{\nu}(s)=\varphi(n),\quad\forall n\in\N.
\end{align*}
This finishes the proof.
\end{proof}

\begin{rmk}
In the general moment problem, the forward shift operator does not necessarily extend to a bounded operator on $\mathcal{H}$. It is only a densely defined symmetric operator, and the proof uses the existence of a self-adjoint extension on a larger space. In our setting, we can avoid those technical issues thanks to the fact that $H$ is bounded.
\end{rmk}

\subsection{Construction of a measure for $2\leq q < \infty$}

Now we will adapt the previous arguments to the case when the operator $B$ from Theorem \ref{Thm_trees}(b) is positive. If $q$ is finite, this is not a Hankel matrix, so the forward shift operator $S$ is not symmetric on the Hilbert space defined by it. We shall need to consider a perturbation of $S$ that will allow us to adapt the proofs to this new setting.
 
We recall the definition of the family of polynomials $(P_n^{(q)})$. Fix $2\leq q<\infty$ and define
\begin{align*}
P_0^{(q)}(x)=1,\quad  P_1^{(q)}(x)=x,
\end{align*}
and
\begin{align*}
P_{n+2}^{(q)}(x)=\left(1+\tfrac{1}{q}\right)xP_{n+1}^{(q)}(x)-\tfrac{1}{q}P_{n}^{(q)}(x),\quad\forall n\geq 0.
\end{align*}
The following lemma is classical and it can be found in the several references given in the introduction. A short proof of it can be given using the results in \cite{HaaSteSzw}.

\begin{lem}\label{Lem_P_n_pd}
For all $s\in[-1,1]$, the function $\varphi:\N\to\R$, given by $\varphi(n)=P_n^{(q)}(s)$, belongs to $\mathcal{R}_+(\mathcal{T}_q)$.
\end{lem}
\begin{proof}
By \cite[Theorem 3.3]{HaaSteSzw}, the Schur norm of the associated multiplier on $\mathcal{T}_q$ is 1. Since $\varphi(0)=1$, Proposition \ref{Prop_norm_pd} implies that this multiplier is positive definite.
\end{proof}

Our goal is to prove the following.

\begin{lem}\label{Lem_var_moment}
Let $\varphi:\N\to\C$. Assume that the operator $B=(B_{i,j})_{i,j\in\N}$ given by
\begin{align*}
B_{i,j} =\sum_{k=0}^{\min\{i,j\}}\frac{1}{q^k}\left(\varphi(i+j-2k)-\varphi(i+j-2k+2)\right)
\end{align*}
is of trace class and positive, and
\begin{align*}
\left|\lim_{n\to\infty}\varphi(2n+1)\right|\leq \lim_{n\to\infty}\varphi(2n).
\end{align*}
Then there exists a positive Borel measure $\mu$ on $[-1,1]$ such that
\begin{align}\label{formula_Prop_var_mom}
\varphi(n)=\int_{-1}^1P_n^{(q)}(s)\,d\mu(s),
\end{align}
for all $n\in\N$.
\end{lem}

For this purpose, we shall need to study another family of polynomials, which is defined in similar fashion. We define $(Q_n^{(q)})_{n\in\N}$ by
\begin{align*}
Q_0^{(q)}(x)=1,\quad  Q_1^{(q)}(x)=\left(1+\tfrac{1}{q}\right)x,
\end{align*}
and
\begin{align}\label{def_Q_n}
Q_{n+2}^{(q)}(x)=\left(1+\tfrac{1}{q}\right)xQ_{n+1}^{(q)}(x)-\tfrac{1}{q}Q_{n}^{(q)}(x),\quad\forall n\geq 0.
\end{align}
We begin by establishing a relation between $(Q_n^{(q)})$ and $(P_n^{(q)})$.

\begin{lem}\label{Lem_Q=P-P}
For all $n\in\N$,
\begin{align}\label{Q=P-P}
P_n^{(q)}(x)-P_{n+2}^{(q)}(x)= \left(1+\tfrac{1}{q}\right)(1-x^2)Q_n^{(q)}(x).
\end{align}
\end{lem}
\begin{proof}
We proceed by induction. First,
\begin{align*}
P_0^{(q)}(x)-P_{2}^{(q)}(x) &= 1-\left(1+\tfrac{1}{q}\right)x^2+\tfrac{1}{q}\\
&= \left(1+\tfrac{1}{q}\right)\left(1-x^2\right),
\end{align*}
and
\begin{align*}
P_1^{(q)}(x)-P_3^{(q)}(x) &= x - \left(1+\tfrac{1}{q}\right)xP_2^{(q)}(x) + \tfrac{1}{q}P_1^{(q)}(x)\\
&= x - \left(1+\tfrac{1}{q}\right)x\left[\left(1+\tfrac{1}{q}\right)x^2-\tfrac{1}{q}\right] + \tfrac{1}{q}x\\
&= \left(1+\tfrac{1}{q}\right)x\left( 1 - \left(1+\tfrac{1}{q}\right)x^2 + \tfrac{1}{q}\right)\\
&= Q_1^{(q)}(x)\left(1+\tfrac{1}{q}\right)(1-x^2).
\end{align*}
This proves \eqref{Q=P-P} for $n=0,1$. Now assume that it holds for $n$ and $n-1$ with $n\geq 1$. Then
\begin{align*}
P_{n+1}^{(q)}(x)-P_{n+3}^{(q)}(x) &= \left(1+\tfrac{1}{q}\right)xP_{n}^{(q)}(x) - \tfrac{1}{q}P_{n-1}^{(q)}(x)\\
&\qquad - \left(1+\tfrac{1}{q}\right)xP_{n+2}^{(q)}(x) + \tfrac{1}{q}P_{n+1}^{(q)}(x)\\
&= \left(1+\tfrac{1}{q}\right)x\left(P_{n}^{(q)}(x)-P_{n+2}^{(q)}(x)\right)\\
&\qquad - \tfrac{1}{q}\left(P_{n-1}^{(q)}(x)-P_{n+1}^{(q)}(x)\right)\\
&= \left(1+\tfrac{1}{q}\right)^2x(1-x^2)Q_{n}^{(q)}(x)\\
&\qquad - \tfrac{1}{q}\left(1+\tfrac{1}{q}\right)(1-x^2)Q_{n-1}^{(q)}(x)\\
&= \left(1+\tfrac{1}{q}\right)(1-x^2)\left(\left(1+\tfrac{1}{q}\right)xQ_{n}^{(q)}(x)- \tfrac{1}{q}Q_{n-1}^{(q)}(x)\right)\\
&= \left(1+\tfrac{1}{q}\right)(1-x^2)Q_{n+1}^{(q)}(x).
\end{align*}
This finishes the proof.
\end{proof}

\begin{lem}\label{Lem_meas_nn+2}
Under the hypotheses of Lemma \ref{Lem_var_moment}, there exists a positive Borel measure $\mu$ on $[-1,1]$ such that
\begin{align}\label{int_Q_n}
\varphi(n)-\varphi(n+2)=\int_{-1}^1 Q_n^{(q)}(s)\,d\mu(s),\quad\forall n\in\N.
\end{align}
\end{lem}
\begin{proof}
We define a scalar product on $c_{00}(\N)$ by
\begin{align*}
\langle f,g\rangle_{\mathcal{H}}=\langle Bf,g\rangle_{\ell_2},
\end{align*}
and let $\mathcal{H}$ be the completion (after taking quotient if necessary) of $c_{00}(\N)$ for this scalar product. Observe that
\begin{align}
B&=\left(1-\tfrac{1}{q}\right)\left(I-\tfrac{1}{q}\tau\right)^{-1}H\label{B=(...)H} \\
&=\left(1-\tfrac{1}{q}\right)\sum_{k\geq 0}\tfrac{1}{q^k}S^kH(S^*)^k,\notag
\end{align}
where $\tau(A)=SAS^*$ and $H=(\varphi(i+j)-\varphi(i+j+2))_{i,j\in\N}$. Recall that the operator $\left(I-\tfrac{1}{q}\tau\right)$ is an isomorphism of $S_1(\ell_2(\N))$. Define
\begin{align*}
\Psi f=\left(1+\tfrac{1}{q}\right)^{-1}\left(S+\tfrac{1}{q}S^*\right)f,\quad\forall f\in c_{00}(\N),
\end{align*}
where $S^*$ stands for the adjoint of $S$ in $\ell_2(\N)$ (not in $\mathcal{H}$). We claim that $\Psi$ is a symmetric operator. Indeed, since $H$ is a Hankel matrix, we have
\begin{align}\label{HS=S*H}
HS=S^* H.
\end{align}
Now observe that \eqref{B=(...)H} can be expressed as
\begin{align*}
\left(1-\tfrac{1}{q}\right) H = B - \tfrac{1}{q}SBS^*,
\end{align*}
so \eqref{HS=S*H} can be written as
\begin{align*}
BS-\tfrac{1}{q}SB=S^*B-\tfrac{1}{q}BS^*,
\end{align*}
or equivalently
\begin{align}\label{B_q_comm}
B\left(S+\tfrac{1}{q}S^*\right)=\left(S+\tfrac{1}{q}S^*\right)^*B.
\end{align}
This shows that
\begin{align*}
\langle \Psi f,g\rangle_{\mathcal{H}}=\langle f,\Psi g\rangle_{\mathcal{H}},\quad\forall f,g\in c_{00}(\N).
\end{align*}
By the same arguments as in the proof of Lemma \ref{Lem_moment_tc}, $\Psi$ extends to a bounded self-adjoint operator on $\mathcal{H}$ with $\|\Psi\|_{\mathcal{H}}\leq 1$. Let $E_\Psi$ denote its spectral measure. Again, we can define a positive measure $\mu$ on $[-1,1]$ by
\begin{align*}
\mu(A)=\langle E_\Psi(A)\delta_0,\delta_0\rangle_{\mathcal{H}},
\end{align*}
for every Borel set $A\subset[-1,1]$. Now we claim that
\begin{align*}
Q_n^{(q)}(\Psi)\delta_0=S^n\delta_0,\quad\forall n\in\N.
\end{align*}
Indeed, for $n=0,1$, this is a short computation. Then, assuming that this holds for $n-1$ and $n$, we get
\begin{align*}
Q_{n+1}^{(q)}(\Psi)\delta_0 &=\left(1+\tfrac{1}{q}\right)\Psi Q_n^{(q)}(\Psi)\delta_0-\tfrac{1}{q}Q_{n-1}^{(q)}(\Psi)\delta_0\\
&= \left(S+\tfrac{1}{q}S^*\right)S^n\delta_0-\tfrac{1}{q}S^{n-1}\delta_0\\
&= S^{n+1}\delta_0,
\end{align*}
and the claim follows by induction. Using this, we obtain
\begin{align*}
\int_{-1}^1 Q_n^{(q)}(s)\,d\mu(s) = \langle Q_n^{(q)}(\Psi)\delta_0,\delta_0\rangle_{\mathcal{H}} = \langle BS^n\delta_0,\delta_0\rangle_{\ell_2},
\end{align*}
which implies by \eqref{B=(...)H} that
\begin{align*}
\left(1-\tfrac{1}{q}\right)^{-1}\int_{-1}^1 Q_n^{(q)}(s)\,d\mu(s) &= \sum_{k\geq 0}\tfrac{1}{q^k}\langle H(S^*)^kS^n\delta_0,(S^*)^k\delta_0\rangle_{\ell_2}\\
&= \langle HS^n\delta_0,\delta_0\rangle_{\ell_2}\\
&= \varphi(n)-\varphi(n+2).
\end{align*}
Thus, the measure $\left(1-\tfrac{1}{q}\right)^{-1}\mu$ satisfies \eqref{int_Q_n}.
\end{proof}

In order to prove that the measure obtained in Lemma \ref{Lem_meas_nn+2} does not have atoms in $\{-1,1\}$, we need to establish some properties of the polynomials $Q_n^{(q)}$.

\begin{lem}\label{Lem_prop_Q_n}
The polynomial functions $Q_n^{(q)}$ satisfy the following:
\begin{itemize}
\item[a)] If $n$ is even (resp. odd), then $Q_n^{(q)}$ is an even (resp. odd) function. 
\item[b)] For all $n\in\N$,
\begin{align*}
Q_n^{(q)}(1)=\sum_{k=0}^n\frac{1}{q^k}.
\end{align*}
\end{itemize}
\end{lem}
\begin{proof}
Both assertions follow by induction. For the first one, simply observe that
\begin{align*}
Q_{n+2}^{(q)}(-s) = -\left(1+\tfrac{1}{q}\right)sQ_{n+1}^{(q)}(-s)-\tfrac{1}{q}Q_{n}^{(q)}(-s),
\end{align*}
and that $Q_0^{(q)}$ and $Q_1^{(q)}$ are even and odd respectively. The second one is given by the identity
\begin{align*}
Q_{n+2}^{(q)}(1) = \left(1+\tfrac{1}{q}\right)Q_{n+1}^{(q)}(1)-\tfrac{1}{q}Q_{n}^{(q)}(1).
\end{align*}
\end{proof}

We shall also need the following properties of the family $(P_n^{(q)})$.

\begin{lem}\label{Lem_prop_P_n}
The family of polynomials $(P_n^{(q)})$ satisfy the following.
\begin{itemize}
\item[a)] For all $s\in[-1,1]$ and all $n\in\N$,
\begin{align*}
|P_n^{(q)}(s)|\leq 1.
\end{align*}
\item[b)] For all $n\in\N$, $P_n^{(q)}(1)=1$ and $P_n^{(q)}(-1)=(-1)^n$.
\item[c)] For all $s\in(-1,1)$,
\begin{align*}
\lim_{n\to\infty}P_n^{(q)}(s)=0.
\end{align*}
\end{itemize}
\end{lem}
\begin{proof}
Let $s\in[-1,1]$, and let $\psi$ be the radial kernel on $\mathcal{T}_q$ given by
\begin{align*}
\psi(x,y)=P_{d(x,y)}^{(q)}(s),\quad\forall x,y\in \mathcal{T}_q.
\end{align*}
By Lemma \ref{Lem_P_n_pd}, $\psi$ is positive definite. Then, for all $n\in\N$,
\begin{align*}
|P_n^{(q)}(s)|\leq \|\psi\|_{cb}=1.
\end{align*}
This proves (a). Property (b) follows by induction. To prove (c), define
\begin{align*}
l_0(s)=\lim_{n\to\infty}P_{2n}^{(q)}(s), & & l_1(s)=\lim_{n\to\infty}P_{2n+1}^{(q)}(s).
\end{align*}
By Corollary \ref{Cor_pd_prod},
\begin{align*}
|l_1(s)|\leq l_0(s).
\end{align*}
On the other hand, we know that the sequence $(P_n^{(q)}(s))$ satisfies
\begin{align*}
P_{2n+2}^{(q)}(s)+\frac{1}{q}P_{2n}^{(q)}(s)=\left(1+\frac{1}{q}\right)sP_{2n+1}^{(q)}(s),
\end{align*}
so taking the limit $n\to\infty$, we get
\begin{align*}
l_0(s)=sl_1(s).
\end{align*}
If $|s|<1$, this can only hold if $l_0(s)=l_1(s)=0$.
\end{proof}

\begin{proof}[Proof of Lemma \ref{Lem_var_moment}]
By Lemma \ref{Lem_meas_nn+2}, we know that there exists a positive Borel measure $\mu$ on $[-1,1]$ such that
\begin{align*}
\varphi(n)-\varphi(n+2)=\int\limits_{[-1,1]} Q_n^{(q)}(s)\,d\mu(s),\quad\forall n\in\N.
\end{align*}
Moreover, using Lemma \ref{Lem_Q=P-P}, we get
\begin{align}
\varphi(n)-\varphi(n+2) &=\int\limits_{[-1,1]} Q_n^{(q)}(s)\,d\mu(s)\notag\\
&=Q_n^{(q)}(-1)\mu(\{-1\})+Q_n^{(q)}(1)\mu(\{1\})\notag\\
&\qquad +\int\limits_{(-1,1)} P_{n}^{(q)}(s)-P_{n+2}^{(q)}(s)\,d\nu(s),\label{phi-phi=P-P}
\end{align}
where $\nu$ is the positive measure on $(-1,1)$ given by
\begin{align*}
d\nu(s)=\left(1+\tfrac{1}{q}\right)^{-1}(1-s^2)^{-1}d\mu(s).
\end{align*}
This looks very similar to \eqref{formula_Prop_var_mom}, but first we need to prove that $\mu\left(\{-1,1\}\right)=0$ and that $P_n^{(q)}\in L_1(\nu)$, so we can write \eqref{phi-phi=P-P} as a sum of two integrals. Define $f_n:(-1,1)\to\R$ by
\begin{align*}
f_n(s)=1-P_{2n}^{(q)}(s),\quad\forall s\in(-1,1).
\end{align*}
By Lemma \ref{Lem_prop_P_n}, for all $s\in(-1,1)$,
\begin{align*}
f_n(s)\geq 0, \quad\forall n\in\N,
\end{align*}
and
\begin{align*}
\lim_{n\to\infty} f_n(s) = 1.
\end{align*}
Therefore, by Fatou's lemma, together with \eqref{phi-phi=P-P},
\begin{align*}
\int\limits_{(-1,1)} 1\,d\nu(s) &\leq \liminf_{n\to\infty}\int\limits_{(-1,1)} f_n\,d\nu(s)\\
&= \liminf_{n\to\infty}\sum_{k=0}^{n-1}\int\limits_{(-1,1)} P_{2k}^{(q)}(s)-P_{2k+2}^{(q)}(s)\,d\nu(s)\\
&= \liminf_{n\to\infty}\sum_{k=0}^{n-1}\varphi(2k)-\varphi(2k+2)-Q_{2k}^{(q)}(1)\mu(\{-1,1\})\\
&=\varphi(0)-\lim_{n\to\infty}\varphi(2n)-\sum_{k\geq 0}Q_{2k}^{(q)}(1)\mu(\{-1,1\}).
\end{align*}
Here we used the fact that $Q_{2k}^{(q)}(-1)=Q_{2k}^{(q)}(1)$, given by Lemma \ref{Lem_prop_Q_n}(a). Moreover, by Lemma \ref{Lem_prop_Q_n}(b), we must have $\mu\left(\{-1,1\}\right)=0$, or else the last sum would be infinite. Observe that this also shows that $\nu$ is a finite measure, from which we deduce that \eqref{phi-phi=P-P} can be written as
\begin{align*}
\varphi(n)-\varphi(n+2)=\int\limits_{(-1,1)} P_{n}^{(q)}(s)\,d\nu(s) - \int\limits_{(-1,1)} P_{n+2}^{(q)}(s)\,d\nu(s).
\end{align*}
As always, we define
\begin{align*}
c_\pm=\frac{1}{2}\lim_{n\to\infty}\varphi(2n)\pm \frac{1}{2}\lim_{n\to\infty}\varphi(2n+1).
\end{align*}
By the dominated convergence theorem, together with Lemma \ref{Lem_prop_P_n}, for all $n\in\N$,
\begin{align*}
\varphi(n)&-(c_++(-1)^nc_-)\\
&=\lim_{k\to\infty}\sum_{j=0}^{k}\varphi(n+2j)-\varphi(n+2j+2)\\
&=\lim_{k\to\infty}\sum_{j=0}^{k}\int\limits_{(-1,1)} P_{n+2j}^{(q)}(s)\,d\nu(s) - \int\limits_{(-1,1)} P_{n+2j+2}^{(q)}(s)\,d\nu(s)\\
&=\int\limits_{(-1,1)} P_{n}^{(q)}(s)\,d\nu(s) - \lim_{k\to\infty}\int\limits_{(-1,1)} P_{n+2k+2}^{(q)}(s)\,d\nu(s)\\
&=\int\limits_{(-1,1)} P_{n}^{(q)}(s)\,d\nu(s).
\end{align*}
Since $c_+,c_-\geq 0$, we can define a new positive measure on $[-1,1]$ by
\begin{align*}
\tilde{\nu}=c_+\delta_{1}+c_-\delta_{-1}+\nu,
\end{align*}
and
\begin{align*}
\int\limits_{[-1,1]} P_{n}^{(q)}(s)\,d\tilde{\nu}(s)=\varphi(n),\quad\forall n\in\N.
\end{align*}
\end{proof}

\subsection{The multi-radial case}
Now we will prove a multi-radial version of Lemma \ref{Lem_var_moment} for products of trees. As before, we fix $2\leq q_1,...,q_N<\infty$ and define
\begin{align*}
X=\mathcal{T}_{q_1}\times\cdots\times\mathcal{T}_{q_N},
\end{align*}
where $\mathcal{T}_{q_i}$ is the $(q_i+1)$-homogeneous tree ($i=1,...,N$). Recall that, given a function $\tilde{\phi}:\N^N \to\C$, we can define the operators $T$ and $T'$ by
\begin{align*}
T_{m,n}=\sum_{I\subset[N]}(-1)^{|I|} \tilde{\phi}(m+n+2\chi^I),\quad\forall m,n\in\N^N,
\end{align*}
and
\begin{align*}
T'= \prod_{i=1}^N\left(1-\frac{1}{q_i}\right)\left(I-\frac{\tau_i}{q_i}\right)^{-1}T.
\end{align*}
Now we consider $N$ families of polynomials $(P_n^{(q_i)})_{n\in\N}$ ($i=1,...,N$) defined as before:
\begin{align*}
P_0^{(q_i)}(x)=1,\quad  P_1^{(q_i)}(x)=x,
\end{align*}
\begin{align}\label{def_P_n^i}
P_{n+2}^{(q_i)}(x)=\left(1+\tfrac{1}{q_i}\right)xP_{n+1}^{(q_i)}(x)-\tfrac{1}{q_i}P_{n}^{(q_i)}(x),\quad\forall n\geq 0.
\end{align}

\begin{lem}\label{Lem_measure_multi}
Let $\tilde{\phi}:\N^N \to\C$ be a function such that the limits
\begin{align*}
l_0&=\lim_{\substack{|n|\to\infty \\ |n|\text{ even}}}\tilde{\phi}(n) & l_1&=\lim_{\substack{|n|\to\infty \\ |n|\text{ odd}}}\tilde{\phi}(n)
\end{align*}
exist. Assume that $T'$ belongs to $\in S_1(\ell_2(\N^N))_+$ and that $|l_1|\leq l_0$. Then there exists a positive Borel measure $\mu$ on $[-1,1]^N$ such that
\begin{align*}
\tilde{\phi}(n_1,...,n_N)=\int\limits_{[-1,1]^N}P_{n_1}^{(q_1)}(t_1)\cdots P_{n_N}^{(q_N)}(t_N)\, d\mu(t_1,...,t_N),
\end{align*}
for all $(n_1,...,n_N)\in\N^N$.
\end{lem}
\begin{proof}
The proof follows the same ideas as the one of Lemma \ref{Lem_var_moment}, hence we shall only sketch it. We construct a Hilbert space $\mathcal{H}$ with scalar product given by
\begin{align*}
\langle f,g\rangle_{\mathcal{H}}=\langle T'f,g\rangle_{\ell_2(\N^N)}, \quad\forall f,g\in c_{00}(\N^N).
\end{align*}
We denote by $\Psi_1,...,\Psi_N$ the family of densely defined operators on $\mathcal{H}$ given by
\begin{align*}
\Psi_i=\left(1+\tfrac{1}{q_i}\right)^{-1}\left(S_i+\tfrac{1}{q_i}S_i^*\right),
\end{align*}
where $S_i$ stands for the forward shift operator on the $i$-th coordinate. Then $\Psi_1,...,\Psi_N$ is a commuting family of operators on $\mathcal{H}$, and by the same arguments as in the proof of Lemma \ref{Lem_moment_tc}, they are bounded, self-adjoint and of norm at most 1. Then the spectral theorem for commuting families of self-adjoint operators allows us to construct a measure $\mu$ on $[-1,1]^N$ by
\begin{align*}
\mu(A_1\times\cdots\times A_N)=\langle E_{\Psi_1}(A_1)\cdots E_{\Psi_N}(A_1)\delta_{(0,...,0)},\delta_{(0,...,0)}\rangle_{\mathcal{H}},
\end{align*}
where $E_{\Psi_i}$ is the spectral measure of the operator $\Psi_i$. Defining a new family of polynomials $(Q_n^{(i)})_{n\in\N}$ ($i=1,...,N$) like in \eqref{def_Q_n}, one sees that
\begin{align*}
Q_n^{(q_i)}(\Psi_i)\delta_{(0,...,0)}=S_i^n\delta_{(0,...,0)},
\end{align*}
for every $n\in\N$. Therefore
\begin{align*}
\int\limits_{[-1,1]^N} Q_{n_1}^{(q_1)}(s_1) &\cdots Q_{n_N}^{(q_N)}(s_N)\,d\mu(s_1,...,s_N) \\
&= \left\langle Q_{n_1}^{(q_1)}(\Psi_1)\cdots Q_{n_N}^{(q_N)}(\Psi_N)\delta_{(0,...,0)},\delta_{(0,...,0)}\right\rangle_{\mathcal{H}} \\
&= \left\langle T'\delta_{(n_1,...,n_N)},\delta_{(0,...,0)}\right\rangle_{\ell_2(\N^N)}.
\end{align*}
Since
\begin{align*}
T'= \prod_{i=1}^N\left(1-\frac{1}{q_i}\right)\sum_{k_1,...,k_N\geq 0}\frac{1}{q^{k_1}\cdots q^{k_N}}S_1^{k_1}\cdots S_N^{k_N}T(S_N^*)^{k_N}\cdots(S_1^*)^{k_1},
\end{align*}
rescaling the measure $\mu$ by the factor $\prod_{i=1}^N\left(1-\frac{1}{q_i}\right)^{-1}$, we get
\begin{align*}
\int\limits_{[-1,1]^N} Q_{n_1}^{(q_1)}(s_1) \cdots Q_{n_N}^{(q_N)}(s_N)&\,d\mu(s_1,...,s_N) \\
&= \left\langle T\delta_{(n_1,...,n_N)},\delta_{(0,...,0)}\right\rangle_{\ell_2(\N^N)}\\
&=\sum_{I\subset[N]}(-1)^{|I|} \tilde{\phi}((n_1,...,n_N)+2\chi^I).
\end{align*}
In particular, by Lemma \ref{Lem_Q=P-P},
\begin{align*}
\sum_{I\subset[N]}(-1)^{|I|} &\tilde{\phi}((n_1,0,...,0)+2\chi^I)\\
&= \int\limits_{\{-1,1\}\times[-1,1]^{N-1}} Q_{n_1}^{(q_1)}(s_1)\,d\mu(s_1,...,s_N)\\
&\quad + \int\limits_{(-1,1)\times[-1,1]^{N-1}} P_{n_1}^{(q_1)}(s_1)-P_{n_1+2}^{(q_1)}(s_1)\,d\nu_1(s_1,...,s_N),
\end{align*}
where $\nu_1$ is given by
\begin{align*}
d\nu_1(s_1,...,s_N)=\left(1+\tfrac{1}{q_1}\right)^{-1}(1-s_1^2)^{-1}d\mu(s,...,s_N).
\end{align*}
Again, using Fatou's lemma, one shows that
\begin{align}
&\int\limits_{(-1,1)\times[-1,1]^{N-1}}  1\,d\nu_1(s_1,...,s_N)\notag\\
&\leq \liminf_{n\to\infty}\sum_{k=0}^{n-1}\int\limits_{(-1,1)\times[-1,1]^{N-1}} P_{2k}^{(q_1)}(s_1)-P_{2k+2}^{(q_1)}(s_1)\,d\nu_1(s_1,...,s_N)\notag\\
&= \liminf_{n\to\infty}\sum_{k=0}^{n-1}\sum_{I\subset[N]}(-1)^{|I|} \tilde{\phi}((2k,0,...,0)+2\chi^I)\label{liminf_telesc}\\
&\qquad\qquad-\sum_{k\geq 0}Q_{2k}(1)\mu\left(\{-1,1\}\times[-1,1]^{N-1}\right).\notag
\end{align}
Observe that
\begin{align*}
\sum_{I\subset[N]}(-1)^{|I|} \tilde{\phi}&((2k,0,...,0)+2\chi^I)\\
&=\sum_{J\subset[N]\setminus\{1\}}(-1)^{|J|} \tilde{\phi}((2k,0,...,0)+2\chi^J)\\
&\qquad-\sum_{J\subset[N]\setminus\{1\}}(-1)^{|J|} \tilde{\phi}((2k+2,0,...,0)+2\chi^J),
\end{align*}
which gives us a telescoping sum in \eqref{liminf_telesc}, and thus
\begin{align*}
\nu_1\left((-1,1)\times[-1,1]^{N-1}\right) &\leq \sum_{J\subset[N]\setminus\{1\}}(-1)^{|J|} \tilde{\phi}((0,...,0)+2\chi^J) - l_0\\
&\qquad\qquad-\sum_{k\geq 0}Q_{2k}(1)\mu\left(\{-1,1\}\times[-1,1]^{N-1}\right).
\end{align*}
This shows that $\nu_1$ is finite and that
\begin{align*}
\mu\left(\{-1,1\}\times[-1,1]^{N-1}\right)=0.
\end{align*}
Repeating this argument for the other coordinates, one shows that
\begin{align*}
\mu\left(\partial[-1,1]^N\right)=0,
\end{align*}
where $\partial[-1,1]^N$ stands for the boundary of $[-1,1]^N$ in $\R^N$. Furthermore,
\begin{align*}
&\sum_{I\subset[N]}(-1)^{|I|} \tilde{\phi}((n_1,...,n_N)+2\chi^I)\\
&=\int\limits_{(-1,1)^N}\prod_{i=1}^N\left(P_{n_i}^{(q_i)}(s_i)-P_{n_i+2}^{(q_i)}(s_i)\right)d\nu(s_1,...,s_N),
%&=\int\limits_{(-1,1)^N}\left(P_{n_1}^{(q_1)}(s_1)-P_{n_1+2}^{(q_1)}(s_1)\right)\cdots \left(P_{n_N}^{(q_N)}(s_N)-P_{n_N+2}^{(q_N)}(s_N)\right)d\nu(s_1,...,s_N),
\end{align*}
where $\nu$ is the finite measure on $(-1,1)^N$ given by
\begin{align*}
d\nu(s_1,...,s_N)=\left(\prod_{i=1}^N\left(1+\tfrac{1}{q_i}\right)^{-1}(1-s_i^2)^{-1}\right)d\mu(s,...,s_N).
\end{align*}
Hence the expression
\begin{align*}
\sum_{k_1,...,k_N\geq 0}\sum_{I\subset[N]}(-1)^{|I|} \tilde{\phi}((n_1+2k_1,...,n_N+2k_N)+2\chi^I)
\end{align*}
can be computed from $N$ telescoping sums, from which we obtain
\begin{align*}
\tilde{\phi}(n_1,...,n_N)-&\left(c_++(-1)^{n_1+\cdots +n_N}c_-\right)\\
&\qquad\qquad=\int\limits_{(-1,1)^N} P_{n_1}^{(q_1)}(s_1)\cdots P_{n_N}^{(q_N)}(s_N)\,d\nu(s_1,...,s_N),
\end{align*}
where $c_\pm=l_0\pm l_1\geq 0$. Therefore
\begin{align*}
\tilde{\phi}(n_1,...,n_N)=\int\limits_{[-1,1]^N} P_{n_1}^{(q_1)}(s_1)\cdots P_{n_N}^{(q_N)}(s_N)\,d\tilde{\nu}(s_1,...,s_N),
\end{align*}
where
\begin{align*}
\tilde{\nu}=c_+\delta_{(1,...,1)}+c_-\delta_{(-1,...,-1)}+\nu.
\end{align*}
\end{proof}

\begin{cor}\label{Cor_meas_prod}
Let $\varphi:\N \to\C$ be a function such that the operator $B=(B_{n,m})_{m,n\in\N^N}$ given by
\begin{align*}
B_{m,n}=\sum_{l\leq m\wedge n}\frac{1}{q_1^{l_1}\cdots q_N^{l_N}}\sum_{k=0}^N\binom{N}{k}(-1)^k \varphi(|m|+|n|-2|l|+2k)
\end{align*}
belongs to $S_1(\ell_2(\N^N))_+$, and such that the following limits
\begin{align*}
l_0&=\lim_{n\to\infty}\varphi(2n), & l_1&=\lim_{n\to\infty}\varphi(2n+1)
\end{align*}
exist and satisfy
\begin{align*}
\left|l_1\right|\leq l_0.
\end{align*}
Then there exists a positive Borel measure $\mu$ on $[-1,1]^N$ such that
\begin{align*}
\varphi(n_1+\cdots +n_N)=\int\limits_{[-1,1]^N}P_{n_1}^{(q_1)}(t_1)\cdots P_{n_N}^{(q_N)}(t_N)\, d\mu(t_1,...,t_N),
\end{align*}
for all $(n_1,...,n_N)\in\N^N$.
\end{cor}
\begin{proof}
Simply apply Lemma \ref{Lem_measure_multi} to the function
\begin{align*}
\tilde{\phi}(n_1,...,n_N)=\varphi(n_1+\cdots+n_N).
\end{align*}
\end{proof}

\section{Proof of the main results}\label{Sect_proof_main}

Sections \ref{Sect_prod_trees} and \ref{Sect_Ham} were devoted to the implications (a)$\Leftrightarrow$(b)$\Rightarrow$(c) of Theorems \ref{Thm_trees} and \ref{Thm_prods}. The remaining implication essentially states that the average of positive definite kernels over a positive measure is again positive definite. Like in the previous sections, we will state the result in the multi-radial context first. Fix $2\leq q_1,...,q_N<\infty$ and define
\begin{align*}
X=\mathcal{T}_{q_1}\times\cdots\times\mathcal{T}_{q_N},
\end{align*}
where $\mathcal{T}_{q_i}$ is the $(q_i+1)$-homogeneous tree, and let $(P_n^{(q_i)})_n$ be the family of polynomials defined in \eqref{def_P_n^i} ($i=1,...,N$).

\begin{lem}\label{Lem_int_pd_multi}
Let $\mu$ be a positive and finite Borel measure on $[-1,1]^N$, and let $\tilde{\phi}:\N^N\to\C$ be given by
\begin{align*}
\tilde{\phi}(n_1,...,n_N)=\int\limits_{[-1,1]^N}P_{n_1}^{(q_1)}(t_1)\cdots P_{n_N}^{(q_N)}(t_N)\, d\mu(t_1,...,t_N).
\end{align*}
Then $\tilde{\phi}:\N^N\to\C$ defines a positive definite multi-radial kernel on $X$.
\end{lem}
\begin{proof}
Since $\mu$ is positive, it is sufficient to prove that, for each $(t_1,...,t_N)\in[-1,1]^N$, the multi-radial kernel given by
\begin{align*}
(n_1,...,n_N)\longmapsto P_{n_1}^{(q_1)}(t_1)\cdots P_{n_N}^{(q_N)}(t_N)
\end{align*}
is positive definite. For each $i\in\{1,...,N\}$, let $\psi_i:\mathcal{T}_{q_i}\times\mathcal{T}_{q_i}\to\C$ be given by
\begin{align*}
\psi_i(x_i,y_i)=P_{d(x_i,y_i)}^{(q_i)}(t_i),\quad\forall x_i,y_i\in \mathcal{T}_{q_i}.
\end{align*}
We know from Lemma \ref{Lem_P_n_pd} that $\psi_i$ is positive definite on $\mathcal{T}_{q_i}$. So by \cite[Theorem D.3]{BroOza}, there is a Hilbert space $\mathcal{H}_i$ and a map $\xi_i:\mathcal{T}_{q_i}\to\mathcal{H}_i$ such that
\begin{align*}
\psi_i(x_i,y_i)=\langle\xi_i(x_i),\xi_i(y_i)\rangle,\quad\forall x_i,y_i\in \mathcal{T}_{q_i}.
\end{align*}
Now define
\begin{align*}
\mathcal{H}=\mathcal{H}_1\otimes\cdots\otimes\mathcal{H}_N,
\end{align*}
and $\xi:X\to\mathcal{H}$ by
\begin{align*}
\xi(x)=\xi_1(x_1)\otimes\cdots\otimes\xi_N(x_N),
\end{align*}
for all $x=(x_1,...,x_N)\in X$. Then
\begin{align*}
\langle \xi(x),\xi(y)\rangle &= \langle\xi_1(x_1),\xi_1(y_1)\rangle\cdots\langle\xi_N(x_N),\xi_N(y_N)\rangle\\
&= \psi_1(x_1,y_1)\cdots \psi_N(x_N,y_N),
\end{align*}
for all $x,y\in X$. Again by \cite[Theorem D.3]{BroOza}, this shows that the kernel
\begin{align*}
(x,y)\longmapsto \psi_1(x_1,y_1)\cdots \psi_N(x_N,y_N)
\end{align*}
is positive definite on $X$, which is exactly what we wanted to prove.
\end{proof}

The following corollary is a direct consequence of Lemma \ref{Lem_int_pd_multi}.

\begin{cor}\label{Cor_suffc_pd_prod}
Let $\varphi:\N\to\C$, and suppose that there is a positive Borel measure $\mu$ on $[-1,1]^N$ such that
\begin{align*}
\varphi(n_1+\cdots+n_N)=\int\limits_{[-1,1]^N}P_{n_1}^{(q_1)}(t_1)\cdots P_{n_N}^{(q_N)}(t_N)\, d\mu(t_1,...,t_N),
\end{align*}
for all $n_1,...,n_N\in\N$. Then $\varphi$ belongs to $\mathcal{R}_+(X)$.
\end{cor}

Now we are ready to give the proofs of the main results of this article.

\begin{proof}[Proof of Theorem \ref{Thm_prods}]
The equivalence (a)$\Leftrightarrow$(b) is given by Corollary \ref{Cor_pd_prod}. Thanks to Corollary \ref{Cor_meas_prod}, we have (b)$\Rightarrow$(c). Finally, (c)$\Rightarrow$(a) is given by Corollary \ref{Cor_suffc_pd_prod}.
\end{proof}

\begin{proof}[Proof of Theorem \ref{Thm_trees}]
For $2\leq q<\infty$, this is a particular case of Theorem \ref{Thm_prods}. For $q=\infty$, (a)$\Leftrightarrow$(b) and (b)$\Rightarrow$(c) are given by Lemmas \ref{Lem_pd_T_infty} and \ref{Lem_moment_tc} respectively. The implication (c)$\Rightarrow$(a) was proven in \cite[\S 3]{HaaKnu}. Another way of seeing it is by noticing that $\mathcal{T}_\infty$ is a median graph and applying Lemma \ref{Lem_pd_median}.
\end{proof}

Observe that our methods also allow us to obtain a similar result for multi-radial kernels. Take $X=\mathcal{T}_{q_1}\times\cdots\times\mathcal{T}_{q_N}$ and $(P_n^{(q_i)})$ as above.

\begin{lem}\label{Lem_equiv_multi_rad}
Let $\tilde{\phi}:\N^N \to\C$, and assume that the limits
\begin{align*}
l_0&=\lim_{\substack{|n|\to\infty \\ |n|\text{ even}}}\tilde{\phi}(n) & l_1&=\lim_{\substack{|n|\to\infty \\ |n|\text{ odd}}}\tilde{\phi}(n)
\end{align*}
exist. Then the following are equivalent:
\begin{itemize}
\item[a)] The function $\tilde{\phi}$ defines a multi-radial positive definite kernel on $X$.
\item[b)] The operator $T'$ defined in \eqref{def_T'_1} belongs to $S_1(\ell_2(\N^N))_+$, and $|l_1|\leq l_0$.
\item[c)] There exists a positive Borel measure $\mu$ on $[-1,1]^N$ such that
\begin{align*}
\tilde{\phi}(n_1,...,n_N)=\int\limits_{[-1,1]^N}P_{n_1}^{(q_1)}(t_1)\cdots P_{n_N}^{(q_N)}(t_N)\, d\mu(t_1,...,t_N),
\end{align*}
for all $(n_1,...,n_N)\in\N^N$.
\end{itemize}
\end{lem}
\begin{proof}
This is a consequence of Corollary \ref{Cor_multirad_pd}, Lemma \ref{Lem_measure_multi} and Lemma \ref{Lem_int_pd_multi}.
\end{proof}

\begin{rmk}
In contrast to Theorem \ref{Thm_prods}, we need to include the existence of the limits $l_0$ and $l_1$ in the hypotheses of Lemma \ref{Lem_equiv_multi_rad} for the three statements to be equivalent. Indeed, take $N=2$ and $\varphi_i\in\mathcal{R}_+(\mathcal{T}_{q_i})$ ($i=1,2$). Then the function $\tilde{\phi}(n_1,n_2)=\varphi_1(n_1)\varphi_2(n_2)$ satisfies (a) and (c), but the limits $l_0,l_1$ do not necessarily exist, so (b) does not really make sense. On the other hand, if $f_1,f_2:\N\to\C$ are any functions, then for $\tilde{\phi}(n_1,n_2)=f_1(n_1)+f_2(n_2)$, the operator $T'$ in (b) is $0$. So, in particular, it belongs to $S_1(\ell_2(\N^N))_+$, but $\tilde{\phi}$ does not even define a Schur multiplier on $\mathcal{T}_{q_1}\times\mathcal{T}_{q_2}$, unless some additional conditions are imposed to the functions $f_1,f_2$.
\end{rmk}

\section{A concrete example}\label{Sect_concr_ex}
Now we turn to the proof of Corollaries \ref{Cor_strct_incl} and \ref{Cor_strct_incl2}. Fix $2\leq q< \infty$. We will focus on the function $\varphi:\N\to\C$ given by
\begin{align}\label{phi-1/q}
\varphi(n)=\begin{cases}
\left(-\frac{1}{q}\right)^\frac{n}{2} & n\text{ even},\\
0 & n\text{ odd},
\end{cases}
\end{align}
with the convention $\left(-\frac{1}{q}\right)^0=1$. Observe that
\begin{align*}
\varphi(n)=P_n^{(q)}(0),\quad\forall n\in\N,
\end{align*}
where $(P_n^{(q)})$ is the family of polynomials defined in \eqref{polynom_P_n}. So by Lemma \ref{Lem_P_n_pd}, we know that $\varphi\in\mathcal{R}_+(\mathcal{T}_q)$.

\begin{lem}
The function $\varphi$ given by \eqref{phi-1/q} does not belong to $\mathcal{R}_+(\mathcal{T}_{q+1})$.
\end{lem}
\begin{proof}
It suffices to prove, by Theorem \ref{Thm_trees}, that the operator
\begin{align*}
B=\sum_{n\geq 0}\frac{1}{(q+1)^n}S^nH(S^*)^n
\end{align*}
is not positive. Here $H=(\varphi(i+j)-\varphi(i+j+2))_{ij}$, and $S$ is the forward shift operator. We shall prove something stronger: For every $\varepsilon>0$, the operator
\begin{align*}
A_{\varepsilon}=\sum_{n\geq 0}\frac{1}{(q+\varepsilon)^n}S^nH(S^*)^n
\end{align*}
is not positive. Observe that
\begin{align*}
\varphi(2n)-\varphi(2n+2)=\left(1+\tfrac{1}{q}\right)\left(-\tfrac{1}{q}\right)^n,
\end{align*}
for all $n\in\N$. Hence,
\begin{align*}
\langle A_{\varepsilon}\delta_1,\delta_1\rangle &= \sum_{n\geq 0}\tfrac{1}{(q+\varepsilon)^n}\langle H(S^*)^n\delta_1,(S^*)^n\delta_1\rangle\\
&= \left(1+\tfrac{1}{q}\right)\left(-\tfrac{1}{q}\right) + \tfrac{1}{q+\varepsilon}\left(1+\tfrac{1}{q}\right)\\
&= \left(1+\tfrac{1}{q}\right)\left(\tfrac{1}{q+\varepsilon}-\tfrac{1}{q}\right),
\end{align*}
and this is strictly negative.
\end{proof}

We conclude that $\varphi$ belongs to $\mathcal{R}_+(\mathcal{T}_q)\setminus\mathcal{R}_+(\mathcal{T}_{q+1})$, which proves Corollary \ref{Cor_strct_incl}.

\begin{lem}
The function $\varphi$ given by \eqref{phi-1/q} does not belong to $\mathcal{R}_+(\mathcal{T}_q\times\mathcal{T}_q)$.
\end{lem}
\begin{proof}
Observe that
\begin{align*}
\varphi(2k)-2\varphi(2k+2)+\varphi(2k+4)=\left(1+\frac{2}{q}+\frac{1}{q^2}\right)\left(-\frac{1}{q}\right)^k,\quad\forall k\in\N.
\end{align*}
Hence, by Theorem \ref{Thm_prods}, we only need to check that the operator $B=(B_{m,n})_{m,n\in\N^2}$ given by
\begin{align*}
B_{m,n}=\begin{cases}
\displaystyle\sum_{l\leq m\wedge n}\frac{1}{q^{|l|}}\left(-\frac{1}{q}\right)^{\frac{|m|+|n|}{2}-|l|} &\text{if } |m|+|n| \text{ is even},\\
0 & \text{ otherwise},
\end{cases}
\end{align*}
is not positive. For this purpose, we shall prove that
\begin{align*}
\left\langle B\left(\delta_{(0,1)}+\delta_{(1,0)}\right), \delta_{(0,1)}+\delta_{(1,0)}\right\rangle < 0.
\end{align*}
Indeed, first observe that for $|m|+|n|$ even,
\begin{align*}
B_{m,n}&=\left(-\frac{1}{q}\right)^{\frac{|m|+|n|}{2}}\sum_{l\leq m\wedge n}(-1)^{|l|}\\
&=\left(-\frac{1}{q}\right)^{\frac{|m|+|n|}{2}}\frac{1}{4}\left(1+(-1)^{m_1\wedge n_1}\right)\left(1+(-1)^{m_2\wedge n_2}\right).
\end{align*}
Thus
\begin{align*}
\left\langle B\delta_{(0,1)}, \delta_{(0,1)}\right\rangle = \left\langle B\delta_{(1,0)}, \delta_{(1,0)}\right\rangle = 0
\end{align*}
and
\begin{align*}
\left\langle B\delta_{(0,1)}, \delta_{(1,0)}\right\rangle = \left\langle B\delta_{(1,0)}, \delta_{(0,1)}\right\rangle = -\frac{1}{q}.
\end{align*}
This shows that $B$ is not positive.
\end{proof}

Therefore $\varphi$ belongs to $\mathcal{R}_+(\mathcal{T}_q)\setminus\mathcal{R}_+(\mathcal{T}_q\times\mathcal{T}_q)$, which proves Corollary \ref{Cor_strct_incl2}.

\section{Median graphs}\label{Sect_median}

This section is devoted to the proof of Proposition \ref{Prop_median}. Recall that a connected graph $X$ is median if 
\begin{align*}
|I(x,y)\cap I(y,z)\cap I(z,x)|=1,\quad\forall x,y,z\in X,
\end{align*}
where
\begin{align*}
I(x,y)=\{u\in X\, :\, d(x,y)=d(x,u)+d(u,y)\}.
\end{align*}
Chepoi proved the following (see \cite[Theorem 6.1]{Che}).

\begin{thm}[Chepoi \cite{Che}]
Median graphs are exactly the 1-skeletons of CAT(0) cube complexes.
\end{thm}

For details on CAT(0) cube complexes, we refer the reader to \cite[\S 2]{GueHig} and the references therein. Thanks to this identification, we may talk about the dimension of a median graph. We say that a median graph $X$ is of dimension $N$ if there exists an isometric embedding from the $N$-dimensional cube $\{0,1\}^N$ into $X$, but there is not such an embedding from $\{0,1\}^{N+1}$. A 1-dimensional median graph is a tree.

A hyperplane in a median graph is an equivalence class of edges under the equivalence relation generated by
\begin{align*}
\{x,y\}\sim \{u,v\}\quad\text{if}\quad\{x,y,v,u\}\text{ is a square}.
\end{align*}
If $H$ is a hyperplane and $\{x,y\}\in H$, we say that $H$ separates $x$ from $y$, and that $\{x,y\}$ crosses $H$. More generally, we say that a path crosses $H$ if one of its edges does. Sageev characterised the distance on a median graph in terms of hyperplanes (see \cite[Theorem 4.13]{Sag}).

\begin{thm}[Sageev \cite{Sag}]\label{thmSag}
Let $x,y$ be two vertices in a median graph $X$, and let $\gamma$ be a geodesic joining $x$ and $y$. Then $\gamma$ crosses every hyperplane separating $x$ from $y$, and it does so only once. Moreover, $\gamma$ does not cross any other hyperplane.
\end{thm}

This says that the distance between two points is exactly the number of hyperplanes separating them.

Niblo and Reeves \cite{NibRee} proved that the distance on a median graph is a conditionally negative kernel. Let $Y$ be a nonempty set. We say that a symmetric function $\psi:Y\times Y\to\R$ is conditionally negative if $\psi(y,y)=0$ for every $y\in Y$, and for every finite sequence $y_1,...,y_n\in Y$ and every $\lambda_1,...,\lambda_n\in\R$ such that $\sum\lambda_i=0$, we have
\begin{align*}
\sum_{i,j=0}^n\lambda_i\lambda_j\psi(y_i,y_j)\leq 0.
\end{align*}
It can be shown that this is equivalent to the fact that there exists a Hilbert space $\mathcal{H}$ and a function $b:Y\to\mathcal{H}$ such that
\begin{align*}
\psi(x,y)=\|b(x)-b(y)\|^2,\quad\forall x,y\in Y.
\end{align*}

\begin{lem}[Niblo--Reeves \cite{NibRee}]\label{d_cond_neg}
Let $X$ be a median graph. Then the edge-path distance on $X$ is a conditionally negative kernel.
\end{lem}

An alternative proof of this fact can be given using the language of spaces with walls. See e.g. \cite[Theorem 12.2.9]{BroOza}. We reproduce the proof in our setting for the reader's convenience.

\begin{proof}[Proof of Lemma \ref{d_cond_neg}]
By Theorem \ref{thmSag}, the distance between two vertices in $X$ is exactly the number of hyperplanes separating them. Let $\mathcal{X}$ be the set of hyperplanes of $X$, and fix $x_0\in X$. For each $x\in X$, let $S_x$ be the set of hyperplanes separating $x$ and $x_0$. Define $b:X\to\ell_2(\mathcal{X})$ by
\begin{align*}
b(x)=\sum_{h\in S_x}\delta_h,\quad\forall x\in X.
\end{align*}
Then
\begin{align*}
\|b(x)-b(y)\|^2=|S_x\Delta S_y|,\quad\forall x,y\in X.
\end{align*}
Fix $x,y\in X$ and let $m$ be the median point between $x,y$ and $x_0$. Then Theorem \ref{thmSag} implies that
\begin{align*}
|S_x\Delta S_y|=d(x,m)+d(y,m)=d(x,y).
\end{align*}
We conclude that
\begin{align*}
d(x,y)=\|b(x)-b(y)\|^2,\quad\forall x,y\in X,
\end{align*}
which implies that $d$ is conditionally negative.
\end{proof}

By Schoenberg's theorem (see e.g. \cite[Theorem D.11]{BroOza}), Lemma \ref{d_cond_neg} implies that, for every $s\in[0,1]$, the function
\begin{align}\label{r^d}
(x,y)\in X \longmapsto s^{d(x,y)}
\end{align}
is positive definite. Moreover, it is not hard to see that on every bipartite graph the function $(x,y)\mapsto(-1)^{d(x,y)}$ is positive definite (see for example \cite[Lemma 2.3]{Ver}). Hence, by composition, the function \eqref{r^d} is positive definite for every $s\in[-1,1]$.

\begin{lem}\label{Lem_pd_median}
Let $X$ be a median graph and let $\mu$ be a finite positive Borel measure on $[-1,1]$. Then the function $\varphi:\N\to\C$ given by
\begin{align}\label{phi=int_median}
\varphi(n)=\int_{-1}^1s^n\,d\mu(s),\quad\forall n\in\N,
\end{align}
belongs to $\mathcal{R}_+(X)$.
\end{lem}
\begin{proof}
Let $x_1,...,x_n$ be a finite sequence of elements of $X$. Let $z_1,...,z_n\in\C$. Then
\begin{align*}
\sum_{i,j=1}^n\varphi(d(x_i,x_j))\overline{z_i}z_j = \int_{-1}^1\sum_{i,j=1}^ns^{d(x_i,x_j)}\overline{z_i}z_j\,d\mu(s),
\end{align*}
and this is non-negative since it is the integral of a non-negative function with respect to a positive measure. This shows that $\varphi\in\mathcal{R}_+(X)$.
\end{proof}

\begin{proof}[Proof of Proposition \ref{Prop_median}]
Let $X$ be a median graph and let $\varphi\in\mathcal{R}_+(\mathcal{T}_\infty)$. By Theorem \ref{Thm_trees}, $\varphi$ is of the form \eqref{phi=int_median} (recall that $P_n^{(\infty)}(s)=s^n$). By Lemma \ref{Lem_pd_median}, this implies that $\varphi\in\mathcal{R}_+(X)$.
\end{proof}

\subsection*{Acknowledgements}
I thank Gilles Pisier for suggesting this problem to me. I am deeply grateful to Adam Skalski for encouraging me to write this paper, for many interesting discussions, and for his careful reading of several different versions of these notes. I also thank Mikael de la Salle for fruitful discussions.\\
This work was supported by the ANR project GAMME (ANR-14-CE25-0004).

\bibliographystyle{plain} 

\bibliography{Bibliography}

\begin{thebibliography}{10}

\bibitem{Arn}
Jean-Pierre Arnaud.
\newblock Fonctions sph\'{e}riques et fonctions d\'{e}finies positives sur
  l'arbre homog\`ene.
\newblock {\em C. R. Acad. Sci. Paris S\'{e}r. A-B}, 290(2):A99--A101, 1980.

\bibitem{Att}
St{\'e}phane Attal.
\newblock Operator and spectral theory.
\newblock {\em Lecture Notes}, 2013.

\bibitem{BroOza}
Nathanial~P. Brown and Narutaka Ozawa.
\newblock {\em {$C^*$}-algebras and finite-dimensional approximations},
  volume~88 of {\em Graduate Studies in Mathematics}.
\newblock American Mathematical Society, Providence, RI, 2008.

\bibitem{Car}
P.~Cartier.
\newblock Harmonic analysis on trees.
\newblock In {\em Harmonic analysis on homogeneous spaces ({P}roc. {S}ympos.
  {P}ure {M}ath., {V}ol. {XXVI}, {W}illiams {C}oll., {W}illiamstown, {M}ass.,
  1972)}, pages 419--424. Amer. Math. Soc., Providence, R.I., 1973.

\bibitem{Che}
Victor Chepoi.
\newblock Graphs of some {${\rm CAT}(0)$} complexes.
\newblock {\em Adv. in Appl. Math.}, 24(2):125--179, 2000.

\bibitem{CohDeM}
Joel~M. Cohen and Leonede De-Michele.
\newblock The radial {F}ourier-{S}tieltjes algebra of free groups.
\newblock In {\em Operator algebras and {$K$}-theory ({S}an {F}rancisco,
  {C}alif., 1981)}, volume~10 of {\em Contemp. Math.}, pages 33--40. Amer.
  Math. Soc., Providence, R.I., 1982.

\bibitem{FigNeb}
Alessandro Fig\`a-Talamanca and Claudio Nebbia.
\newblock {\em Harmonic analysis and representation theory for groups acting on
  homogeneous trees}, volume 162 of {\em London Mathematical Society Lecture
  Note Series}.
\newblock Cambridge University Press, Cambridge, 1991.

\bibitem{GueHig}
Erik Guentner and Nigel Higson.
\newblock Weak amenability of {$\rm CAT(0)$}-cubical groups.
\newblock {\em Geom. Dedicata}, 148:137--156, 2010.

\bibitem{HaaSteSzw}
U.~Haagerup, T.~Steenstrup, and R.~Szwarc.
\newblock Schur multipliers and spherical functions on homogeneous trees.
\newblock {\em Internat. J. Math.}, 21(10):1337--1382, 2010.

\bibitem{HaaKnu}
Uffe Haagerup and S\o{}ren Knudby.
\newblock A {L}\'{e}vy-{K}hinchin formula for free groups.
\newblock {\em Proc. Amer. Math. Soc.}, 143(4):1477--1489, 2015.

\bibitem{Ham}
Hans Hamburger.
\newblock \"{U}ber eine {E}rweiterung des {S}tieltjesschen {M}omentenproblems.
\newblock {\em Math. Ann.}, 81(2-4):235--319, 1920.

\bibitem{Mur}
Gerard~J. Murphy.
\newblock {\em {$C^*$}-algebras and operator theory}.
\newblock Academic Press, Inc., Boston, MA, 1990.

\bibitem{NibRee}
Graham Niblo and Lawrence Reeves.
\newblock Groups acting on {${\rm CAT}(0)$} cube complexes.
\newblock {\em Geom. Topol.}, 1:1--7, 1997.

\bibitem{Pel}
Vladimir~V. Peller.
\newblock {\em Hankel operators and their applications}.
\newblock Springer Monographs in Mathematics. Springer-Verlag, New York, 2003.

\bibitem{Pis2}
Gilles Pisier.
\newblock {\em Similarity problems and completely bounded maps}, volume 1618 of
  {\em Lecture Notes in Mathematics}.
\newblock Springer-Verlag, Berlin, expanded edition, 2001.
\newblock Includes the solution to ``The Halmos problem''.

\bibitem{Pis}
Gilles Pisier.
\newblock {\em Introduction to operator space theory}, volume 294 of {\em
  London Mathematical Society Lecture Note Series}.
\newblock Cambridge University Press, Cambridge, 2003.

\bibitem{Pow}
S.~C. Power.
\newblock Hankel operators on {H}ilbert space.
\newblock {\em Bull. London Math. Soc.}, 12(6):422--442, 1980.

\bibitem{Sag}
Michah Sageev.
\newblock Ends of group pairs and non-positively curved cube complexes.
\newblock {\em Proc. London Math. Soc. (3)}, 71(3):585--617, 1995.

\bibitem{Ver}
Ignacio Vergara.
\newblock Radial {S}chur multipliers on some generalisations of trees.
\newblock {\em Studia Math.}, 249(1):59--109, 2019.

\end{thebibliography}

\end{document}